\documentclass[10pt]{amsart}

\usepackage{soul}

\usepackage{standard}
\usepackage{needspace}
\usepackage[margin=3.5cm]{geometry}
\usepackage{color}

\newcommand{\Tbstar}{{}^{b}T^*}

\usepackage{perpage,hyperref,enumerate}

\newcommand{\e}{\epsilon}
\newcommand{\tr}{\operatorname{tr}}
\newcommand{\re}{\mathbb{R}}
\newcommand{\mc}[1]{\mathcal{#1}}
\newcommand{\ta}{\tilde{a}}

\newcommand{\Mtilde}{\widetilde{M}}
\newcommand{\TM}{T^*\!M}

\numberwithin{equation}{section}
\numberwithin{lemma}{section}

\DeclareMathOperator{\Diff}{Diff}

\DeclareMathOperator{\loc}{loc}
\DeclareMathOperator{\comp}{comp}

\newcommand{\hyp}{\mathcal{H}}
\newcommand{\E}{\mathcal{E}}
\newcommand{\gl}{\mathcal{G}}
\newcommand{\tgl}{\widetilde{\mathcal{G}}}

\newcommand{\rprime}{R'}
\newcommand{\rpprime}{R''}

\usepackage{amsmath}
\usepackage{graphicx}
\usepackage{amsfonts}
\usepackage{amssymb}%
\usepackage{latexsym}

\usepackage[show]{ed}

\newcommand{\beq}{\begin{equation}}
\newcommand{\eeq}{\end{equation}}
\newcommand{\beqs}{\begin{equation*}}
\newcommand{\eeqs}{\end{equation*}}
\newcommand{\bit}{\begin{itemize}}
\newcommand{\eit}{\end{itemize}}
\newcommand{\ben}{\begin{enumerate}}
\newcommand{\een}{\end{enumerate}}
\newcommand{\bal}{\begin{align}}
\newcommand{\eal}{\end{align}}
\newcommand{\bals}{\begin{align*}}
\newcommand{\eals}{\end{align*}}
\newcommand{\bse}{\begin{subequations}}
\newcommand{\ese}{\end{subequations}}
\newcommand{\bpr}{\begin{proposition}}
\newcommand{\epr}{\end{proposition}}
\newcommand{\bre}{\begin{remark}}
\newcommand{\ere}{\end{remark}}
\newcommand{\bpf}{\begin{proof}}
\newcommand{\epf}{\end{proof}}
\newcommand{\ble}{\begin{lemma}}
\newcommand{\ele}{\end{lemma}}
\newcommand{\bco}{\begin{corollary}}
\newcommand{\eco}{\end{corollary}}
\newcommand{\bex}{\begin{example}}
\newcommand{\eex}{\end{example}}
\newcommand{\bth}{\begin{theorem}}
\newcommand{\enth}{\end{theorem}}

\newcommand{\noi}{\noindent}

\newcommand{\R}{\mathbb{R}}

\newcommand{\cH}{{\mathcal H}}

\newcommand{\cT}{{\mathcal T}}

\newcommand{\bx}{x}

\newcommand{\bxi}{\xi}

\newcommand{\eps}{\varepsilon}

\newcommand{\ri}{i}
\newcommand{\rd}{d}

\newcommand{\Rea}{\mathbb{R}}

\newcommand{\diff}[2]{\frac{\rd #1}{\rd #2}}
\newcommand{\pdiff}[2]{\frac{\partial #1}{\partial #2}}

\newcommand{\OR}{{\Omega_R}}

\newcommand{\GR}{{\Gamma_R}}

\newcommand{\gu}{\nabla u}

\newcommand{\gvb}{\overline{\nabla v}}

\newcommand{\vb}{\overline{v}}

\newcommand{\HoDk}{{H^1_{0,D}(\domain_R)}}

\newcommand{\HoDkkg}{{H^1_{k,A,\coeffn}(\domain_R)}}

\newcommand{\tendi}{\rightarrow \infty}
\newcommand{\tendo}{\rightarrow 0}

\newcommand*{\N}[1]{\left\|#1\right\|}

\newcommand{\vertiii}[1]{{\left\vert\kern-0.25ex\left\vert\kern-0.25ex\left\vert #1 
    \right\vert\kern-0.25ex\right\vert\kern-0.25ex\right\vert}}

\usepackage{hyperref}
\definecolor{myblue}{rgb}{0,0,0.6}
\hypersetup{colorlinks=true,
linkcolor=myblue,citecolor=myblue,filecolor=myblue,urlcolor=myblue}
\allowdisplaybreaks[4]

\newcommand{\ton}{\text{ on }}
\newcommand{\tin}{\text{ in }}
\newcommand{\tfa}{\text{ for all }}

\newcommand{\tand}{\text{ and }}
\newcommand{\tst}{\text{ such that }}

\newcommand{\domain}{\Omega}

\newcommand{\Creg}{{C_{H^2}}}
\newcommand{\Ccont}{{C_{\rm cont}}}
\newcommand{\Cint}{{C_{\rm int}}}
\newcommand{\Ctint}{{\widetilde{C}_{\rm int}}}
\newcommand{\CTR}{{C_{\rm DtN}}}
\newcommand{\CtTR}{{\widetilde{C}_{\rm DtN}}}

\definecolor{escol}{rgb}{0,0,0.8}
\definecolor{estcol}{rgb}{0,0.8,0}

\newcommand{\hFEM}{{h_{\rm FEM}}}

\newcommand{\coeffA}{A}
\newcommand{\coeffn}{\nu}

\newcommand{\matrixstyle}[1]{{\mathsf{#1}}}
\newcommand{\matrixA}{\matrixstyle A}
\newcommand{\matrixI}{\matrixstyle I}
\newcommand{\matrixB}{\matrixstyle B}

\newcommand{\trace}{\gamma}

\newcommand{\CR}{\mathcal{C}}
\newcommand{\absorb}{\alpha}

\newcommand{\ballR}{B(0,R)}

\title[Optimal constants in nontrapping resolvent estimates]{Optimal constants in nontrapping resolvent estimates and applications in numerical analysis}
\author{Jeffrey Galkowski}
\address{Department of Mathematics, University College London, London, UK}
\email{j.galkowski@ucl.ac.uk}
\author{Euan A.~Spence}
\address{Department of Mathematical Sciences, University of Bath, Bath, UK}
\email{E.A.Spence@bath.ac.uk}
\author{Jared Wunsch}
\address{Department of Mathematics, Northwestern University, Evanston, IL, USA}
\email{jwunsch@math.northwestern.edu}
\date{\today}

\begin{document}
\begin{abstract}
We study the resolvent for nontrapping obstacles on manifolds with
Euclidean ends. It is well known that for such manifolds, the outgoing
resolvent satisfies $\|\chi R(k) \chi\|_{L^2\to L^2}\leq
C{k}^{-1}$ for ${k}>1$, but the constant $C$ has been
little studied. We show that, for high frequencies, the
constant is bounded above by $2/\pi$ times the length of the
longest generalized bicharacteristic of $|\xi|_g^2-1$ remaining in the
support of $\chi.$  We show that this estimate is optimal in the case
of manifolds without boundary.
We then explore the implications of this result for the numerical analysis of the Helmholtz equation.
\end{abstract}

\maketitle

\section{Introduction}
Let $(M,g)$ be a manifold with Euclidean ends and $\Omega\Subset M$ an
obstacle with smooth boundary.   Assume that all
Melrose-Sj\"ostrand generalized bicharacteristics (i.e., geodesics) escape to
infinity.
Let $\Delta_{\Omega,g}$
be the Dirichlet realization of the Laplacian on
$M\setminus \Omega$. It is well known that for any
$\chi \in \CR_c^{\infty}(M)$, there exists a constant $C>0$ and
${k}_0>0$ so that
$$
\|\chi (-\Delta_{\Omega,g}-{k}^2-i0)^{-1}\chi \|_{L^2(M\setminus \Omega)\to L^2(M\setminus \Omega)}\leq C|k|^{-1},\qquad k>k_0.
$$
In this paper we study how the constant $C>0$ depends on the classical dynamics on $((M,g),\Omega)$.

Suppose that there are no geodesics tangent to
$\partial\Omega$ to infinite order and let $\varphi_t:S^*\!M\to S^*\!M$
denote the Melrose--Sj\"ostrand generalized bicharacteristic flow~\cite[Section 24.3]{Hormander:v3}. Next, fix $R_1>0$ so that $\Omega\subset B(0,R_1)$ and $(M,g)$ is Euclidean outside
$B(0,R_1)$. Then define for any $R \geq R_1,$
\beq\label{eq:L}
L(g,\Omega,R):=\inf\big\{ t>0\mid \varphi_t(S^*_{B(0,R)}M)\cap S^*_{B(0,R)}M=\emptyset\big\}.
\eeq
(We will omit the $\Omega$ from the notation when $\Omega=\emptyset.$)

In the statement of the following estimates, we will use a family of
Sobolev spaces with appropriate semiclassical scaling, using the
global definition for $s \in \RR,$
\beq\label{eq:Hsnorm}
\smallnorm{u}_{H^s(M\setminus \Omega)}^2 =  \ang{(-\Lap_{\Omega,g}+k^2)^{s} u,u}.
\eeq

\begin{theorem}\label{theorem:main}
Let $(M,g)$ be a manifold with Euclidean ends with $g\in \CR^{1,1}.$ Suppose that
$\Omega\Subset M$ has smooth boundary, $g$ is $\CR^\infty$ near $\Omega$, and $\Omega$ is nowhere tangent to the geodesic
flow to infinite order. Assume the generalized geodesic flow on
$M\backslash \Omega$ is nontrapping.
Then for every $R>R_1$, $\chi\in
\CR_c^{\infty}(B(0,R);[0,1]),$ 
there exists ${k}_0>0$ so that
for ${k}>{k}_0$, 
\begin{equation}\label{upperbound}
\|\chi (-\Delta_{\Omega,g}-{k}^2-i0)^{-1}\chi\|_{L^2(M\setminus \Omega)\to L^2(M\setminus \Omega)}\leq
\frac{2 L(g,\Omega,R)}{\pi{k}}.
\end{equation}
More generally, for $0\leq s\leq 2$,
\begin{equation}\label{Sobolevupperbound}
\|\chi (-\Delta_{\Omega,g}-{k}^2-i0)^{-1}\chi\|_{L^2(M\setminus \Omega)\to H^s(M\setminus \Omega)}\leq
\frac{2^{\frac{s}{2}+1} L(g,\Omega,R)}{\pi} {k}^{s-1}.
\end{equation}
The constant ${k}_0$ may be chosen uniformly as $g$ varies
within a sufficiently small open neighborhood in $\CR^{2,\alpha}$ ($\alpha>0$)
of a given nontrapping metric in $\CR^{2,\alpha}$ (where all metrics are
taken smooth near $\pa\Omega$).

Conversely, when $\Omega=\emptyset$ and $g\in \CR^\infty$, for every $R'>R_1$ with $B(0,R')\subset \{\chi \equiv 1\}$, there is ${k}_0>0$ so that for ${k}>{k}_0$,
\begin{equation}\label{lowerbound}
\|\chi(-\Delta_{g}-{k}^2-i0)^{-1}\chi\|_{L^2(M\setminus \Omega)\to L^2(M\setminus \Omega)}\ \geq \frac{2L(g,R')}{\pi {k}}.
\end{equation}
\end{theorem}
Notice that in the interior of $M\setminus \Omega$,
the Melrose-Sj\"ostrand flow is equal to the geodesic flow except that, on
the unit cotangent bundle, the speed of the flow is 2 rather than
1. Thus, in the case where $\Omega=\emptyset$, the theorem states that
the growth of the resolvent as a map on $L^2$ as ${k}\to \infty$ is controlled above and below by $1/\pi$ times the length of the longest geodesic contained entirely in $B(0,R_1)$. 

\subsection{More general operators}\label{sec:moregeneral}

In \S\ref{sec:EDPsetup} below, we recall that several important
  physical applications of the Helmholtz equation involve an operator
  that differs slightly from $-\Delta_{\Omega,g},$ namely the
divergence form operator $-\sum \partial_i g^{ij} \partial_j$ which is
self-adjoint with respect to the Euclidean volume form on $\RR^n.$ In
order to deal with this and similar operators, we will prove a
slightly more general result.
\begin{theorem}\label{thm:2}
  Let $(M,g)$ be a manifold with Euclidean ends, $g\in \CR^{1,1}.$
  Suppose that $\Omega\subset M$ has smooth boundary, $g$ is
  $\CR^{\infty}$ in a neighborhood of $\Omega$, and $\Omega$ is
  nowhere tangent to the geodesic flow to infinite order. Assume the generalized geodesic flow on
$M\backslash \Omega$ is nontrapping. Suppose that
  $P(g)\in \Diff^2(M)$ so that \beq\label{eq:convert1}
  P(g)+\Delta_g=\sum_j L_j\partial_{x_j} +L \eeq where
  $L_j,L \in \CR^{0,\alpha}_c(B(0,R_1)),$ for some $\alpha>0$. Suppose
  further that $\nu$ is a density on $M$ so that $P_\Omega(g)$ is self
  adjoint with respect to $L^2(M\setminus \Omega;\nu)$ where
  $P_{\Omega}(g)$ is the Dirichlet realization of $P(g)$. Let
\beq\label{eq:weightednorm}
\smallnorm{u}_{H_\nu^s(M\setminus \Omega)}^2 :=  \ang{(P_{\Omega}(g)+{k}^2)^{s} u,u}_{L^2(M\setminus \Omega;\nu)}.
\eeq
Then for every $R>R_1$, $\chi\in
\CR_c^{\infty}(B(0,R);[0,1]),$ there exists ${k}_0>0$ so that
for ${k}>{k}_0$ and $0\leq s\leq 2$,
\begin{equation}\label{Sobolevupperbound2}
\|\chi (P_{\Omega}(g)-{k}^2-i0)^{-1}\chi\|_{L_\nu^2(M\setminus \Omega)\to H_\nu^s(M\setminus \Omega)}\leq
\frac{2^{\frac{s}{2}+1} L(g,\Omega,R)}{\pi} {k}^{s-1}.
\end{equation}
The constant ${k}_0$ may be chosen uniformly as $g$ varies
within sufficiently small open neighborhoods of in $\CR^{2,\alpha}$ ($\alpha>0$)
of a given nontrapping metric in $\CR^{2,\alpha}$ and $L_j,$
$L$ vary in small neighborhoods in $\CR^{0,\alpha}$ (where all $g,L_j, L$ are
taken smooth near $\pa\Omega$ and the subset is assumed to be contained in a small open neighborhood in $\CR^N$ for some sufficiently large $N$ near $\pa\Omega$).
\end{theorem}

\subsection{Motivation from, and applications to,  numerical analysis.}\label{sec:motivation}

In the last few years, there has been growing interest from the
numerical-analysis community in proving bounds on solutions of the
Helmholtz equation where the constants are explicit in the metric
(i.e, the coefficients); see \cite{BrGaPe:15, Ch:16, BaChGo:17,
  OhVe:18, SaTo:17, MoSp:17, GrPeSp:18, GrSa:18}.  
Almost all of these previously-obtained bounds
used variants of the Morawetz commutator $\bx\cdot \nabla$, and thus are restricted to star-shaped domains and certain classes of coefficients (although this class includes discontinuous coefficients; see, e.g., \cite{GrPeSp:18}); the exception are the 1-d bounds in \cite{SaTo:17}, which use the fact that the solution of the Helmholtz equation with piecewise-constant coefficients in 1-d can be expressed in terms of the solution of a linear system of algebraic equations.

This interest from the numerical-analysis community is
because the analysis of \emph{any} numerical method for solving the
Helmholtz equation with variable coefficients requires a resolvent
estimate, and if the constant in the resolvent estimate is not given
explicitly in terms of the coefficients, then the numerical analysis
will not be explicit in the coefficients. For example, having proved a bound explicit in the coefficients, \cite{Ch:16, GrSa:18} were then concerned with analyzing standard
finite-element methods applied to Helmholtz problems with variable
coefficients, and \cite{BrGaPe:15, Ch:16, BaChGo:17, OhVe:18} were concerned with
designing and analyzing methods tailored to the coefficients.

The resolvent estimate \eqref{Sobolevupperbound2} will therefore be a fundamental ingredient in the numerical analysis of variable-coefficient Helmholtz problems in nontrapping scenarios. 
In \S\ref{sec:FEM} we illustrate this fact by proving an error estimate for the finite-element method applied to the variable-coefficient Helmholtz equation posed in the exterior of a nontrapping Dirichlet obstacle, with this estimate explicit in $k$, the coefficients, and the parameters of the discretization; see Theorem \ref{thm:Emain} below. 
The key point is that Theorem \ref{thm:Emain} shows how the condition on the discretization for the error estimate to hold depends on the length of the longest ray, showing that this condition becomes more restrictive as the length of the longest ray grows. 

In \S\ref{sec:FEM} we also briefly outline the implications of
the estimate \eqref{Sobolevupperbound2} for (i) preconditioning finite-element discretizations of the Helmholtz equation (see Remark \ref{rem:precondition}), and (ii) ``uncertainty quantification" of the Helmholtz equation (see Remark \ref{rem:UQ}).

\subsection*{Acknowledgements}
The first author was supported by NSF Postdoctoral Research Fellowship DMS-1502661 and 
thanks Maciej Zworski for helpful conversations.
The second author was supported by EPSRC grant
EP/R005591/1.
The third author was partially
supported by NSF grant DMS--1600023.  The authors are grateful to an
anonymous referee for helpful comments on the manuscript.

\section{Manifolds with Euclidean Ends}
\label{s:eucEnds}
We now define the notion of a manifold with Euclidean ends. Note that the canonical example of a manifold with Euclidean ends is the space $\RR^n$ with a metric $g$ so that $I-g$ has compact support. In order to allow more general topologies, we define the general notion of a manifold with Euclidean ends.
\begin{definition}\label{def:euc}
$(M,g)$ is an $n$-dimensional \emph{manifold with Euclidean ends} if
it is a non-compact, complete Riemannian manifold such that
\begin{itemize}
\item there exists a function $r\in \CR^{\infty}(M;\mathbb R)$ such that
the sets $\{r\leq c\}$ are compact for all $c$, and
\item there exists $R_1>0$ such that $\{r\geq R_1\}$ is the disjoint union
of finitely many components, each of which is isometric to $\mathbb R^n\setminus B(0,R_1)$
with the Euclidean metric, and the pullback of $r$ under the isometry
is the Euclidean norm.
\end{itemize}
\end{definition}
The connected components of $\{r\geq R_1\}$ are called the \emph{infinite ends} of $M$. The notation $B(0,R)$ is defined for $R\geq R_1$ and has the following meaning:
$$
B(0,R):=\{r<R\}.
$$

\subsection{The outgoing resolvent on manifolds with Euclidean ends}

We now review (following the treatment in \cite[Section 4.2]{ZwScat})
some properties of the outgoing resolvent on a manifold
with Euclidean ends. Let $E_i$, $i=1,\dots m$ be the infinite ends of
$M$. and let $R_0({k})$ denote the free resolvent on
$\RR^n$. That is
$R_0({k}):L^2_{\comp}(\RR^n)\to L^2_{\loc}(\RR^n)$ is the
meromorphic continuation of $(-\Delta-{k}^2)^{-1}$ from the half
plane $\Im {k}>0$. Define
$\tilde{R}_0({k}):L^2_{\comp}(M)\to L^2_{\loc}(M)$ by
\begin{equation}
\label{e:free}
\tilde{R}_0({k})f:=\sum_{i=1}^m 1_{E_i}R_0({k}) 1_{E_i}f.
\end{equation}
Next, let $\chi_i\in \CR_c^{\infty}(M;[0,1])$, $i=0,\dots 3$ so that
\begin{equation}
\label{e:cutoffs}
\begin{gathered}
\chi_i\equiv 1\text{ on }\supp \chi_{i-1},\qquad 
 \supp(1-\chi_i)\subset M\setminus \overline{B(0,R_1)},\qquad\supp \chi_i \subset B(0,R)
 \end{gathered}
\end{equation}
for some $R>R_1$. Then for ${k}_0\in \mathbb{C}$ with $\Im {k}_0\gg 1$,  let 
$$
\begin{gathered}
Q({k},{k}_0):=Q_0({k})+Q_1({k}_0)\\
Q_0({k}):=(1-\chi_0)\tilde{R}_0({k})(1-\chi_1),\qquad Q_1({k}):=\chi_2(-\Delta_g-{k}_0^2)^{-1}\chi_1.
\end{gathered}
$$
following the proof in~\cite[Section 4.2]{ZwScat}, we can write 
\begin{equation}
\label{e:resForm}
R({k}):=(P_\Omega(g)-{k}^2)^{-1}=Q({k},{k}_0)(I+K({k},{k}_0)\chi_3)^{-1}(I-K({k},{k}_0)(1-\chi_3)).
\end{equation}
where $K({k},{k}_0):L^2(M)\to L_{\comp}^2(M)$ and in particular, $(1-\chi_3)K({k},{k}_0)=0$. This implies that $(1-\chi_3)R({k})$ lies in the image of $\tilde{R}_0({k})$.

\section{The case of a manifold without boundary}
To illustrate the methods we begin by proving \eqref{upperbound} in
the case of a manifold without boundary. The idea of the proof is
identical when the boundary is non-empty, but the proofs of
propagation statements are more involved. Thus we take $\pa
M=\emptyset$ throughout this section.

\subsection{Defect measures}

We will argue by contradiction. Let $h={k}^{-1}$ and $P(h):=h^2P_{\Omega}(g)-1$. We also write $L^2$ for $L^2_\nu(M\setminus \Omega)$ and $H^s$ for $H^s_\nu(M\setminus \Omega)$.

\begin{remark}
We will sometimes write $\|\cdot\|_{H_h^s}$ for
$h^s\|\cdot\|_{H^s};$ since $\smallnorm{\cdot}_{H^s}$ has a
semiclassical scaling built into it, this means that
$$
\norm{u}^2_{H_h^s} = \smallang{(h^2\Lap_{\Omega,g}+1)^s u,u}.
$$
\end{remark} 

Note that if $\chi_0$, $\chi_1\in \CR_c^\infty(M;[0,1])$ and $\chi_1\equiv 1$ on $\supp \chi_0$, then 
\begin{align*}
\frac{\|\chi_0 (P-i0)^{-1}\chi_0f\|_{L^2}}{\|f\|_{L^2}}&= \frac{\|\chi_0 \chi_1(P-i0)^{-1}\chi_1\chi_0f\|_{L^2}}{\|f\|_{L^2}}\leq \frac{\|\chi_0 \chi_1(P-i0)^{-1}\chi_1\chi_0f\|_{L^2}}{\|\chi_0f\|_{L^2}}\\
&\leq \frac{\|\chi_1(P-i0)^{-1}\chi_1\chi_0f\|_{L^2}}{\|\chi_0f\|_{L^2}}\leq \|\chi_1(P-i0)^{-1}\chi_1\|_{L^2\to L^2}.
\end{align*}
In particular, 
$$
\|\chi_0 (P-i0)^{-1}\chi_0\|_{L^2\to L^2}\leq \|\chi_1(P-i0)^{-1}\chi_1\|_{L^2\to L^2}.
$$
and, since $R>R_1$, \emph{we may assume without loss of generality that $\chi\equiv 1$ on $B(0,R_1)$}.

  If \eqref{upperbound} fails, there
exists a sequence of discrete values of $h=h_j \downarrow 0$ and a sequence
$0\neq f(h)\in L^2$ such that 
$$
\|\chi(P-i0)^{-1}\chi hf(h)\|_{L^2}=1
$$
and
$$
\lim_{h\to 0}\frac{\|\chi (P-i0)^{-1}\chi
  hf(h)\|_{L^2}}{\|f(h)\|_{L^2}}=  M\geq 2 L(g, R).
$$
(note that we allow $M=\infty$.)
There exists $R'<R$ such that we
still have $\supp \chi \subset B(0,R'),$ hence we have the strict inequality $M>L(g, R').$
Let
$$ 
u(h)=(P-i0)^{-1}\chi hf\in L^2_{\loc}.
$$
so that $\|\chi u\|_{L^2}=1$.

\begin{lemma}
\label{l:l2loc}
For all $\tilde{\chi}\in \CR_c^{\infty}(M)$, there exists $C>0$ so that $\|\tilde{\chi} u\|_{L^2}\leq C$.
\end{lemma}
\begin{proof}
Let $\chi_i\in \CR_c^{\infty}(M;[0,1])$, $i=0, 1$ so that $\chi\equiv
1$ on $\supp \chi_1$ and $\chi_1\equiv 1$ on $\supp \chi_0$, and
  $\chi_0\equiv 1$ on $B(0,R_1)$. Then
$$
(1-\chi_0)P(1-\chi_1)=\sum_i(1-\chi_0)(-h^2\Delta_{\RR^n}-1)1_{E_i}(1-\chi_1)
$$ 
where $E_i$ denotes the $i^{\text{th}}$ Euclidean end of $M$. 

Then,
\begin{align*}
P(1-\chi_1)u&=\sum_i(1-\chi_0)1_{E_i}(-h^2\Delta_{\RR^n}-1)1_{E_i}(1-\chi_1)u\\
&=(1-\chi_1)\chi hf-\sum_i(1-\chi_0)[-h^21_{E_i}\Delta_{\RR^n}1_{E_i},\chi_1]u\\
&=(1-\chi_1)\chi hf-\sum_i[-h^21_{E_i}\Delta_{\RR^n}1_{E_i},\chi_1]u.
\end{align*}
Next extend $u 1_{E_i}$ by $0$ to a function $v_i$ on $\RR^n$ so that $v_i\equiv u1_{E_i}$ on $\RR^n\setminus B(0,R_1)$.  Then
$$
(-h^2\Delta_{\RR^n}-1)(1-\chi_1)1_{E_i}v_i=1_{E_i}(1-\chi_1)\chi hf-[h^2\Delta_{\RR^n},1_{E_i}\chi_1]v_i.
$$
Since by~\eqref{e:resForm} $v_i$ is $h^{-1}$ outgoing,
$$
(1-\chi_1)1_{E_i}v_i=h^{-2}R_0(h^{-1})(1_{E_i}(1-\chi_1)\chi hf-[h^2\Delta_{\RR^n},1_{E_i}\chi_1]v_i).
$$
In particular, since $\chi\equiv 1$ on $\supp \chi_1$,
$$
\|(1-\chi_1)1_{E_i}v_i\|_{L^2}\leq C(\|f\|_{L^2}+\|[\Delta_{\RR^n},1_{E_i}\chi_1]v_i\|_{H^{-1}})\leq C(\|f\|_{L^2}+\|\chi u\|_{L^2}).
$$
Now, 
$$
(1-\chi_1)u= (1-\chi_1)\sum_i 1_{E_i}v_i.
$$
Therefore,
$$
\|(1-\chi_1)u\|_{L^2}\leq C(\|f\|_{L^2}+\|\chi u\|_{L^2}).
$$

Let $\tilde{\chi}\in \CR_c^{\infty}(M;[0,1])$. Then using again that $\chi\equiv 1$ on $\supp \chi_1$,
\begin{align*}
\limsup_{h\to 0}\| \tilde{\chi}u\|_{L^2}&\leq \limsup_{h\to 0}(\|\tilde{\chi}(1-\chi)u\|_{L^2}+\|\tilde{\chi}\chi u\|_{L^2})\\
&\leq\limsup_{h\to 0}( \|\tilde{\chi}(1-\chi_1)u\|_{L^2}+\|\chi u\|_{L^2})\\
&\leq C\limsup_{h\to 0}(\|f\|_{L^2}+\|\chi u\|_{L^2})\leq  C\Big(\frac{1}{2L(g,\rprime)}+1\Big)
\end{align*}
completing the proof of the lemma.
\end{proof}

By Lemma~\ref{l:l2loc}, $u$ is uniformly bounded in $L^2_{\loc}$ and, taking subsequences, we may assume that $u$ has defect
measure $\mu$, $\chi f$ has defect measure $\alpha$, and $u$ and $\chi f$ have
joint defect measure $\mu^j;$ in other words, for $h=h_j$ in this
chosen subsequence, and for every $a \in \CR_c^\infty(T^*\!M),$ 
\begin{equation}\label{defectmeasures}
\begin{aligned}
\lim_{h\downarrow 0} \ang{a(x,hD) u, u} &= \int a \, d\mu\\
\lim_{h\downarrow 0} \ang{a(x,hD) \chi f, \chi f} &= \int a \, d\alpha\\
\lim_{h\downarrow 0} \ang{a(x,hD)\chi f, u} &= \int a \, d\mu^j.
\end{aligned}
\end{equation}
Note that we use a quantization procedure that sends symbols with compact support in $x$ to operators with compactly supported kernel.

The Cauchy-Schwarz inequality gives us a simple inequality satisfied by
these three measures.  We use the notation
$$
\mu(a) \equiv \int a\, d\mu
$$
for the pairing of a function and a measure.
\begin{lemma}\label{lemma:CS}
For any $a \in \CR_c^0(T^*\!M;\re)$.
$$
\smallabs{\mu^j(a)}\leq \sqrt{\mu(\smallabs{a})}\sqrt{\alpha(\smallabs{a})}.
$$
\end{lemma}
\begin{proof}
If $b=\sqrt{a}$ is in $\CR_c^{\infty}(T^*\!M),$
\begin{align*}
\ang{a(x,hD)\chi f,u}&=\ang{b^2(x,hD)\chi f,u}\\
 &= \ang{b(x,hD) \chi f, b(x,hD)u} +O(h)\\
&\leq \sqrt{\ang{b(x,hD) \chi f, b(x,hD) \chi f}} \sqrt{\ang{b(x,hD) u, b(x,hD)
  u}}+O(h)\\
&\leq \big( \int b^2 \, d\alpha\big)^{1/2}\big( \int b^2 \, d\mu\big)^{1/2}+o(1),
\end{align*}
which proves the desired inequality by letting $h \to 0.$

Next, let $a\in \CR_c^{\infty}(T^*\!M)$ with $a\geq 0$, $\psi_i\in \CR_c^{\infty}(T^*M;[0,1])$, $i=0,1$ have $\psi_i \equiv 1$ in a neighborhood of $\supp a$ and $\psi_1\equiv 1$ in a neighborhood of $\supp \psi_0$. 
Then we apply the previous result to $a_\ep=\psi_0^2(a+\ep\psi_1)$ and let $\ep \downarrow 0$ to see that $|\mu^j(a)|\leq \sqrt{\mu(a)}\sqrt{\alpha(a)}.$

Now, letting $0\leq a\in \CR_c^0(T^*\!M)$, $0\leq a_n\in \CR_c^{\infty}(T^*\!M)$ with $a_n\to a$ uniformly. Then we have
$$\mu^j(a_n)\to \mu^j(a),\qquad \mu(a_n)\to \mu(a),\qquad \alpha(a_n)\to \alpha(a)$$
In particular, 
$$
|\mu^j(a)|\leq \sqrt{\mu(a)}\sqrt{\alpha(a)},\qquad 0\leq a\in \CR_c^0(T^*\!M).
$$
Next, for $a\in \CR_c^0(T^*\!M;\re)$, write $a=a_+-a_-$ with $0\leq a_{\pm}\in \CR_c^0(T^*\!M)$. Then,
\begin{align*}
|\mu^j(a)|&\leq |\mu^j(a_+)|+|\mu^j(a_-)|\leq \sqrt{\mu(a_+)}\sqrt{\alpha(a_+)}+\sqrt{\mu(a_-)}\sqrt{\alpha(a_-)}\\
&\leq \sqrt{\mu(|a|)}\sqrt{\alpha(|a|)}.
\end{align*}
\end{proof}

Since $Pu=h\chi f$, for $a\in \CR_c^{\infty}(T^*\!M;\R)$,   
\begin{equation}\label{basiccommutator}
\begin{aligned}
ih^{-1}\langle [P, a(x,hD)]u,u\rangle&=ih^{-1}(\langle a(x,hD)u,Pu\rangle- \langle a(x,hD)Pu,u\rangle\\
&=2\Im \langle a(x,hD)\chi f,u\rangle.
\end{aligned}
\end{equation}
Sending $h\to 0$ yields
\begin{equation}\label{transport1}
\mu(H_pa)=2\Im \mu^j(a).
\end{equation}

For the proof of the following standard result, we use the formula~\eqref{e:resForm} relating $R({k}):=(P_{\Omega}(g)-{k}^2-i0)^{-1}$ to the free outgoing resolvent on $\RR^n$ and refer the reader to, e.g., 
\cite[Proposition 3.5]{Burq:2002}.
\begin{lemma}
\label{l:incoming}
Let 
$$
\mc{I}:=\Big\{\rho \in S^*\!M\,\big| \,\bigcup_{t\geq0}\varphi_{-t}(\rho)\cap \supp \chi =\emptyset\Big\}
$$
be the directly incoming set. Then $\mu(\mc{I})=0$. 
\end{lemma}

Next, we show that $\mu$ is supported on the characteristic variety and that $u$ is oscillating at frequency roughly $h^{-1}$.
\begin{lemma}
\label{l:h-osc}
The measure $\mu$ is supported on $S^*\!M$. In addition, for $b\in S^1$,
$$
\lim_{h\to 0}\|b(x,hD)\chi u\|_{L^2}^2=\mu(|b|^2\chi^2).
$$
\end{lemma}
\begin{proof}
Suppose that $a\in \CR_c^{\infty}(T^*\!M)$ with $a\equiv 0$ in a
neighborhood of $|\xi|_g=1$. Then there exists $E\in \Psi^{-2}$
compactly supported so that 
$$a(x,hD)=EP +O_{L^2_{\loc}\to L_{\comp}^2}(h^\infty).$$
Therefore, 
$$
\langle a(x,hD)u,u\rangle =\langle Eh\chi f,u\rangle +O(h^\infty)\to 0.
$$
Hence, $\mu(a)=0$. In particular, $\supp \mu\subset S^*\!M$. 

Fix $\psi\in \CR_c^{\infty}(T^*\!M)$ with $\psi \equiv 1$ on $\supp \chi\cap \{|\xi|^2_g\leq 2\}$. Then, there exists $E\in \Psi^{0}$ compactly supported so that
$$\chi(x)b(x,hD)^*b(x,hD)\chi(x)(1-\psi(x,hD))=EP +O_{L^2_{\loc}\to L_{\comp}^2}(h^\infty).$$
In particular, 
\begin{align*}
\|b(x,hD)\chi u\|_{L^2}^2&=\langle \chi(x)b(x,hD)^*b(x,hD)\chi(x)u,u\rangle \\
&=\langle \chi(x)b(x,hD)^*b(x,hD)\chi(x)\psi(x,hD)u,u\rangle \\
&\qquad+\langle \chi(x)b(x,hD)^*b(x,hD)\chi(x)(1-\psi(x,hD))u,u\rangle\\
&=\langle \chi(x)b(x,hD)^*b(x,hD)\chi(x)\psi(x,hD)u,u\rangle +\langle hE\chi f,u\rangle+O(h^\infty)\\
&\to \mu(|b|^2\chi^2\psi)=\mu(|b|^2\chi^2).
\end{align*}
\end{proof}

\subsection{H\"older continuous metrics}
We now make the necessary adjustments to allow the metric $g$ to be
H\"older continuous.   We refer the reader to \cite[Chapter 3, Section
11]{Ta:00} for an analogous account of propagation of singularities
for the wave equation as well as a review of the history of
propagation of singularities theorems for operators with rough
coefficients.  Stronger results (i.e., with weaker regularity
hypotheses) are probably possible in line with the
work of Burq--Zuily \cite[Remark 3.3]{BuZw:15}, but low regularity is
not our main focus here.

\begin{lemma}
\label{l:lowReg}
Let $g_0, L_0,L_{j,0}$ satisfy the hypotheses of Theorem~\ref{thm:2}.
Suppose that $g(h)\in \CR^{1,\alpha}$ and $L(h), L_j(h) \in \CR^{0,\alpha}$ for some $\alpha>0$ satisfy
$$
\begin{aligned}
  \lim_{h\to 0}\|g(h)-g_0\|_{C^{1,\alpha}}&=0,\\
  \lim_{h\to 0}\|L(h)-L_0\|_{C^{0,\alpha}}&=0,\\
    \lim_{h\to 0}\|L_j(h)-L_{j,0}\|_{C^{0,\alpha}}&=0,\quad
    j=1,\dots,n.
    \end{aligned}
 $$
Then the measure $\mu$ is supported in $S^*\!M$, for $b\in S^1$,
\begin{equation}
\label{e:lowRegL2}
\mu(|b|^2\chi^2)=\lim_{h\to 0}\|Op_h(b)\chi u\|^2_{L^2}
\end{equation}
and for $a\in \CR_c^{\infty}(T^*\!M)$, 
\begin{equation}
\label{e:lowRegFlow}
\mu(H_pa)=2\Im \mu^j(a).
\end{equation}
where $p=|\xi|^2_{g_0}-1$.
\end{lemma}
This lemma (together with the results of Section~\ref{sec:volterra} below) suffices
to prove our estimate \emph{provided the bicharacteristic flow is
  unique}.  Note that the lemma only requires the hypothesis that $g
\in \CR^{1,\alpha}$ for $\alpha>0$ and that this regularity
suffices to prove existence of solutions to Hamilton's equations. However,
it is not adequate for proving uniqueness.  Hence
Theorem~\ref{theorem:main} is only stated for $\alpha=1,$ which is
sufficient to guarantee the existence of a well-defined single-valued
bicharacteristic flow.  (A quantitative result using the dynamics of
the multi-valued flow for $\alpha<1$ would be an interesting direction
for further research.)

To prove Lemma~\ref{l:lowReg}, we will need a more general set of symbol classes.  Let $r\in \mathbb{N}$, $0<\alpha<1$, $0\leq \rho<1$. We say that $p(x,\xi)\in \CR^{r,\alpha}S_\rho^m$ if
$$
 \|D_{\xi}^\beta p(\cdot,\xi)\|_{C^{r,\alpha}}\leq C_\beta h^{-r\rho} \langle \xi\rangle^{m-|\beta|}.
$$
We say $p\in S_\rho^m$ if $p\in \bigcap_r \CR^{r,\alpha}S_\rho^m$. 
We need the following boundedness property for operators with symbol in $\CR^{r,\alpha}S^m$. 
\begin{lemma}
\label{l:lowRegBound}
For $r\geq 0$ and $\alpha>0$, the map  $Op_h:\CR^{r,\alpha}S^m\to \mc{L}(H_h^m,L^2)$ with $p\mapsto p(x,hD)$ is continuous with norm bounded independently of $h$. 
\end{lemma}
\begin{proof}
It is enough to show that $Op_h:\CR^{r,\alpha}S^0\to \mc{L}(L^2,L^2)$ is continuous with norm bounded independently of $h$. For this, we rescale to $h=1$. That is, let $T_h:L^2(\re^n)\to L^2(\re^n)$ be given by $(T_hu)(x)=h^{\frac{n}{4}}u(h^{\frac{1}{2}}x)$. Then $T_h$ is unitary and 
$$Op_h(a)u=T_h^*Op_1(a_h)T_hu$$
where
$$
a_h(x,\xi)=a(h^{\frac{1}{2}}x,h^{\frac{1}{2}}\xi).
$$
Now, for $a\in \CR^{r,\alpha}S^m$, 
$$\|D_{\xi}^\beta a_h(\cdot,\xi)\|_{\CR^{r,\alpha}}\leq C_\beta h^{\frac{|\beta|}{2}} \langle h^{1/2}\xi\rangle^{-|\beta|}.$$
In particular, $a_h$ is uniformly bounded in
$\CR^{r,\alpha}S^m$. Therefore, by~\cite[Chapter 13, Theorem 9.1]{Tay3} $Op_1(a_h):L^2\to L^2$ uniformly in $h$ with norm bounded by a finite sum of $\CR^{r,\alpha}S^0$ seminorms. Since $T_h$ is unitary, this completes the proof. 
\end{proof}
\begin{lemma}
\label{l:commutator}
Let $a\in \CR_c^{\infty}(T^*\!M)$, $m\in \re$, and $0<\alpha<1$
$$
\CR^{1,\alpha}S^m\ni p\mapsto h^{-1}[p(x,hD),a(x,hD)]\in \mc{L}(L^2,L^2)
$$
is continuous with norm bounded independently of $h$. 
\end{lemma}
\begin{proof}
Let $p\in \CR^{1,\alpha}S^m$ and choose $\rho \in ({2}/{(2+\alpha)}, 1)$
so that $(1+{\alpha}/{2})\rho> 1$. Let $\psi \in \CR_c^{\infty}(\re)$ with
$\psi \equiv 1$  on $[-2,2]$. Define $p_h(x,\xi)=(\psi(h^\rho
\lvert D_x\rvert )p)(x,\xi)$. Then, $p_h\in S^m_\rho$. Moreover,  
\begin{gather*}
|D_{x} D_\xi ^\beta p_h(x,\xi)|\leq C_\beta \langle \xi\rangle^{m-|\beta|},\\
 |D_{x}^\gamma D_{\xi}^\beta p_h(x,\xi)|\leq C_\beta h^{-\rho(|\gamma|-1)}\langle \xi\rangle^{m-|\beta|},\quad |\gamma|>1.
\end{gather*}
Note that the symbol classes $S^m_\rho$ have a symbol calculus and, in particular, 
\begin{equation}
\label{e:smoothTerm}
h^{-1}[Op_h(p_h),a(x,hD)]= \frac{1}{i}Op_h(\{p_h,a\})+O(h^{1-\rho})_{L^2\to L^2}.
\end{equation}
Next, note that by the characterization of of H\"older spaces by
Littlewood--Paley decomposition \cite[Chapter 13, Theorems 8.1, 8.2]{Tay3},
$$
\| (p-p_h)(\cdot,\xi)\|_{\CR^{0,\alpha/2}}\leq h^{\rho(1+\frac{\alpha}{2})}\|p(\cdot,\xi)\|_{\CR^{1,\alpha}}=o(h)\|p(\cdot,\xi)\|_{C^{1,\alpha}}
$$
In particular, $p-p_h\in o(h)\CR^{0,\alpha/2}S^2$ and hence
$$
h^{-1}[Op_h(p-p_h),a(x,hD)]=o(1)_{L^2\to L^2}.
$$
Combining this with~\eqref{e:smoothTerm} completes the proof.
\end{proof}

\begin{proof}[Proof of Lemma~\ref{l:lowReg}]
We start with the proof of~\eqref{e:lowRegL2}. Note that $P=Op_h(p_0)+hOp_h(p_1)$ where $p_0(h)=|\xi|_{g(h)}^2-1\in \CR^{1,\alpha}S^2$ and $p_1(h)\in \CR^{0,\alpha}S^{1}$ uniformly as $h\to 0$.

 Fix $\psi \in \CR_c^{\infty}(\re)$ with $\psi \equiv 1$  on $[-2,2]$ and
 let $\e>0$. Let $p_{i,\e}(x,\xi)=(\psi(\e |D_x|)p_i)(x,\xi).$ Then,
 $p_{i,\e}\in S^{2-i}$ and again by Littlewood--Paley,
 $$
 \|D_\xi^\beta (p_0-p_{0,\e})(\cdot,\xi)\|_{\CR^{0,\alpha}}\leq C_\beta \e\langle \xi\rangle^{2-|\beta|}.
 $$  
 In particular, for $\e>0$ small and $h<h_0$ small, $p_\e$ is elliptic on $||\xi|_g-1|\geq C\e$. Therefore, for $a\in S^0(T^*\!M)$ supported away from $p=0$, there is $\e>0$, $h_0$ small enough so that $p_{0,\e}$ is elliptic on $\supp a$. Hence there is $E_\e\in \Psi^{-\infty}$ (uniformly for $\e>0$ small) so that 
 $$
 a(x,hD)=E_\e Op_h(p_{0,\e})+O(h^\infty)_{\Psi^{-\infty}}.
 $$
 In particular, 
 $$
 a(x,hD)=E_\e (P-Op_h(p_{0}-p_{0,\e})-hOp_h(p_1))+O(h^\infty)_{\Psi^{-\infty}}
 $$
 Therefore, 
 \begin{align*}
 \mu(a)&=\lim_{h\to 0}\langle E_\e Pu,u\rangle -\langle E_\e [Op_h(p_{0}-p_{0,\e})+hOp_h(p_1)]u,u\rangle\\
 &=\lim_{h\to 0} \langle E_\e h\chi f,u\rangle +O(\e)=O(\e)
 \end{align*}
 Since the left-hand side does not depend on $\e$, sending $\e\to 0$, $\mu(a)=0$ and hence $\supp\mu\subset S^*\!M$. 
 
Next, fix $K>0$ large enough so that for $h<h_0$, $\{|\xi|_{g(h)}\leq
2\}\subset \{|\xi|\leq K\}.$ We apply the previous argument with
$a(x,hD)=\chi(x) b(x,hD)^*b(x,hD) \chi(x) (1- \psi(K^{-1}|hD|)).$ In particular, there exists $E_\e\in \Psi^{0}$ compactly supported so that
$$\chi(x) b(x,hD)^*b(x,hD) \chi(x) (1-\psi(|hD|_g))=E_\e P +O_{L^2_{\loc}\to L_{\comp}^2}(\e).$$
Therefore,
\begin{align*}
\lim_{h\to 0}\|b(x,hD)\chi u\|_{L^2}^2&=\lim_{h\to 0}\langle
                                         \chi(x) b(x,hD)^*b(x,hD) \chi(x) u,u\rangle \\
&=\lim_{h\to 0}\langle \chi(x) b(x,hD)^*b(x,hD) \chi(x)\psi(K^{-1}|hD|)u,u\rangle \\
&\qquad +\lim_{h\to 0}\langle \chi(x) b(x,hD)^*b(x,hD) \chi(x) (1-\psi(K^{-1}|hD|))u,u\rangle\\
&=\lim_{h\to 0}\langle \chi(x) b(x,hD)^*b(x,hD) \chi(x)\psi(K^{-1}|hD|)u,u\rangle\\
&\qquad\qquad +\lim_{h\to 0}\langle hE_\e f,u\rangle+O(\e)\\
&= \mu(|b|^2\chi^2\psi(|\xi|_g))+O(\e)=\mu(|b|^2\chi^2)+O(\e).
\end{align*}
Again, since the left hand side is independent of $\e$, this proves~\eqref{e:lowRegL2}.

We now prove~\eqref{e:lowRegFlow}. Let $a\in \CR_c^{\infty}(T^*\!M;\R)$. Then, 
$$
\frac{i}{h}\langle [P,a(x,hD)]u,u\rangle = i[(a(x,hD)u,h\chi f)-(ha(x,hD)\chi f,u)]\to 2\Im \mu^j(a).
$$
Therefore, it remains to show that 
\begin{equation}
\label{e:claimSquid}\frac{i}{h}\langle [P,A]u,u\rangle \to \mu(H_{p}a).
\end{equation}
First, observe that by Lemma~\ref{l:lowRegBound} and the fact that 
$$
 \|D_\xi^\beta (p_1-p_{1,\e})(\cdot,\xi)\|_{\CR^{0,\frac{\alpha}{2}}}\leq C_\beta \e^{\frac{\alpha}{2}}\langle \xi\rangle^{1-|\beta|},
 $$
 we obtain
$$\|ha(x,hD)Op_h(p_1-p_{1,\e})\|+\|hOp_h(p_1-p_{1,\e})a(x,hD)u\|_{L^2}\leq Ch\e^{\frac{\alpha}{2}}.$$
Therefore, 
\begin{equation}
\label{e:l.o.t}
\|[hOp_h(p_1),a(x,hD)u]\|_{L^2}\leq Ch\e^{\frac{\alpha}{2}}\|u\|_{L^2}+O_\e(h^2).
\end{equation}

Next, observe that by Lemma~\ref{l:commutator}, 
\begin{equation}
\label{e:approxEst}
\|[Op_h(p_0-p_{0,\e}),a(x,hD)]u\|_{L^2}\leq C\e h\|u\|_{L^2}. 
\end{equation}
Finally, with $q_\e:=\lim_{h\to 0}p_{0,\e}$,
\begin{equation}
\label{e:smoothTerm2}
\lim_{h\to 0}\langle \frac{i}{h}[Op_h(p_{0,\e}),a(x,hD)]u,u\rangle = \mu(H_{q_\e}a).
\end{equation}
Combining~\eqref{e:l.o.t},~\eqref{e:approxEst}, and~\eqref{e:smoothTerm2} gives
$$
\Big|\lim_{h\to 0}\frac{i}{h}\langle [P,A]u,u\rangle -\mu(H_{q_{\e}}a)\Big|\leq C\e^{\frac{\alpha}{2}}
$$
Since $p_0\in \CR^{1,\alpha}S^2$ uniformly in $h$, $H_{q_\e}\to H_{p}$. In particular, sending $\e\to 0$ and applying the dominated convergence theorem gives~\eqref{e:claimSquid}.
\end{proof}

\subsection{Appearance of the Volterra Operator}\label{sec:volterra}
To complete the proof of \eqref{upperbound} in the boundaryless case,
we prove the following measure-theoretic proposition. This lemma will
be modified slightly in the case of a manifold with boundary to
account for the fact that the generalized bicharacteristic flow is not
generated by a vector field.

\begin{proposition}
\label{p:volterra}
Suppose that $\mu$, $\mu^j$, and $\alpha$ are Radon measures on
$T^*\!M$ with $\alpha$ finite and $\mu$ finite on compact subsets of
$T^*\!M$. Let $X$ be a continuous vector field on $T^*\!M$ and
$\varphi_t:T^*\!M\to T^*\!M$ be given by $\varphi_t(q)=\exp(tX)(q)$.
Let $\Sigma\subset T^*\!M$ be a hypersurface transverse to $X$ so that
the map $F:\re\times \Sigma\ni (t,q)\mapsto \varphi_t(q)\in T^*\!M$ is
homeomorphism onto its image. 
Suppose that $\mu(F((-\infty,0)\times \Sigma))=0$ and for $a\in \CR_c^0(T^*\!M;\re)$,
\begin{equation}
\label{e:sqrt}
|\mu^j(a)|\leq \sqrt{\mu(|a|) \alpha(|a|)}.
\end{equation}
Furthermore, for $\psi\in \CR_c^1(\re)$, $a\in \CR_c^0(\Sigma)$, let $f_{\psi,a}(F(t,q))=\psi(t)a(q)$ and assume that 
\begin{equation}
\label{e:derMeasure}
\mu(f_{\partial_t\psi,a})=2\Im \mu^j(f_{\psi,a}).
\end{equation}
Then,
$$
\mu(F([0,L]\times \Sigma))\leq \frac{4L^2}{\pi^2}\alpha(F([0,L]\times \Sigma))).
$$
\end{proposition}

\begin{proof}
For $0\leq a\in \CR_c^0(T^*\!M)$, define the Radon measures $\mu^\sharp_{a}$, $\Im \mu^{j,\sharp}$, and $\alpha_a^\sharp$ on $\re$ by
$$
\mu^\sharp_{a}(\psi)=\mu(f_{\psi,a}),\qquad\mu^{j,\sharp}_{a}(\psi)=\mu^j(f_{\psi,a}),\qquad \alpha^\sharp_{a}(\psi)=\mu(f_{\psi,a}).
$$
\begin{lemma}\label{lemma:AC}
The measure $\mu^\sharp$ is absolutely continuous with respect to Lebesgue measure and 
$$
\mu_a^\sharp =\dot \mu^\sharp_adt
$$ 
with $|\dot \mu^\sharp_a|\leq C\sup|a|.$ 
\end{lemma}
\begin{proof}
First, observe that for $\psi\in \CR_c^1(\re)$,~\eqref{e:derMeasure} implies
$$
\mu^\sharp_a(\partial_t\psi)=2\Im \mu^j(f_{\psi,a}).
$$
Let $[b,c]\subset \re$, $0\leq \chi_\e\in \CR_c^{\infty}(\re;[0,1])$ with $\chi\equiv 1$ on $[b,c]$ and $\supp \chi_\e\subset (b-\e,c+\e)$. Define
$$
\psi_\e(t)=-\int \chi_\e(s)ds+\int_{-\infty}^t\chi_\e(s)ds.
$$
 Then, since $\mu((-\infty,0)\times \Sigma)=0$, $\mu^j(F((-\infty,0)\times \Sigma))=0$ and hence 
 \begin{align*}
| \mu^\sharp_a(\chi_\e)|&=2|\Im \mu^j(1_{F([0,\infty)\times \Sigma)} f_{\psi_\e,a})|\\
&\leq C\sup|a|\sup|\psi_\e|\\
&\leq C\sup|a|\Big|\int\chi_\e ds\Big|\leq C\sup|a|(c-b+2\e).
 \end{align*}
Since $\chi_\e\to 1_{[b,c]}$, we have by the dominated convergence theorem,
 $$
 |\mu^\sharp_a([a,b])|\leq C\sup|a|(c-b).
 $$
In particular, $\mu^\sharp_a$ is absolutely continuous with respect to Lebesgue measure since the intervals generate the Borel sets on $\re$. Moreover, its Radon--Nikodym derivative is linear in $a$ and  bounded by $C\sup|a|$. 
\end{proof}

By~\eqref{e:sqrt}, since $\mu_a^\sharp$ is absolutely continuous with respect to Lebesgue measure, $\mu_a^{j,\sharp}$ is also. Let
$\dot \mu_a^\sharp$, $\dot\mu_a^{j,\sharp}$, and $\dot\alpha_a^\sharp$ be the Radon--Nikodym derivatives of $\mu_a^\sharp$, $\mu_a^{j,\sharp}$, and $\alpha_a^\sharp$ respectively with respect to Lebesgue measure. In particular,
\begin{align*}
\mu_a^{\sharp}&=\dot\mu_a^\sharp dt,&\Im \mu_a^{j,\sharp}&= \Im \dot\mu_a^{j,\sharp}dt,&
\alpha_a^\sharp&=\dot \alpha_a^\sharp dt+\lambda
\end{align*}
where $\lambda\perp dt$. Now, by~\eqref{e:sqrt}, 
$$
\frac{1}{2r}\Big|\mu_a^{j,\sharp}([t-r,t+r])\Big|\leq \frac{1}{2r}\sqrt{\mu_a^\sharp([t-r,t+r])}\sqrt{\alpha_a^\sharp([t-r,t+r])}
$$
So, by the Lebesgue differentiation theorem~\cite[Theorem 3.22]{Fol}, for Lebesgue a.e. $t$,
\begin{equation}
\label{e:LebSqrt}
|\dot \mu_a^{j,\sharp}(t)|\leq \sqrt{\dot\mu_a^\sharp(t)}\sqrt{\dot\alpha_a^\sharp(t)}.
\end{equation}

\begin{lemma}
\label{l:cts}
With $\dot \mu_a^\sharp(t)$ as above, $\dot\mu_a^\sharp(t)$ is continuous and 
\begin{equation}
\label{e:intIneq}
\dot \mu_a^\sharp(t)\leq 2\int_0^{\max(t,0)}|\Im \dot \mu_a^{j,\sharp}(s)|ds\leq 2\int_0^t\sqrt{\dot\mu_a^\sharp(s)}\sqrt{\dot\alpha_a^\sharp(s)}ds.
\end{equation}
\end{lemma}
\begin{proof}
By~\eqref{e:derMeasure}, for $\psi\in \CR_c^1(\re)$,
\begin{equation}
\label{e:absContDer}
\int \psi'(t) \dot \mu_a^{\sharp}(t)dt=2\int \psi(t)\Im \dot \mu_a^{j,\sharp}(t)dt.
\end{equation}
Let $\chi\in \CR_c^{\infty}(\re;[0,1])$ with $\int \chi(s)ds\equiv 1$ and define $\chi_{\e,t'}(s)=\e^{-1}\chi(\e^{-1}(s-t'))$ and 
$$
\psi_{\e,t'}(s)=-1+\int_{-\infty}^t\chi_{\e,t'}(s)ds.
$$
Then by~\eqref{e:absContDer}, together with the fact that $\mu_a^{j,\sharp}(t)=0$ for $t<0$,
$$
\int \chi_{\e,t'}(t)\dot\mu_a^\sharp(t)dt=2\int_0^\infty \psi_{\e,t'}(t)\Im \dot \mu_a^{j,\sharp}(t)dt.
$$
Applying the Lebesgue differentiation theorem on the left and dominated convergence on the right, we find for Lebesgue a.e. $t'$ 
$$
\dot\mu_a^\sharp(t')=-2\int_0^{\max(t',0)}\Im \dot\mu_a^{j,\sharp}(t)dt.
$$
Note that this implies $\dot{\mu}_a^\sharp(t')$ is continuous. Applying~\eqref{e:LebSqrt} gives~\eqref{e:intIneq}.
\end{proof}

\begin{lemma}
\label{l:intIneq}
Let $0\leq b(t),c(t)\in L^2$ with $b(t)$ continuous and
$$b^2(t)\leq 2\int_0^tb(s)c(s)ds.$$
Then
\begin{equation}\label{foo}
  b(t)\leq \int_0^tc(s)ds.\end{equation}
\end{lemma}
\begin{proof}
  Supposing \eqref{foo} to be false, there exists $\e>0$ and
    $t> 0$ such that
    $$
      b(t)\geq \int_0^tc(s)ds+\e;
      $$
  hence we may define
$$
t_0:=\inf\Big\{t>0\mid \int_0^tc(s)ds+\e\leq  b(t)\Big\}.
$$

Then, 
\begin{align*}\Big[\int_0^{t_0}c(s)ds+\e\Big]^2&=b^2(t_0)\leq 2\int_0^{t_0}b(s)c(s)ds\\
&\leq 2\int_0^{t_0}\int_0^s c(r) c(s)drds+2\e\int_0^{t_0}c(s)ds\\
&=\Big[\int_0^{t_0}\int_0^sc(r)c(s)drds+\int_0^{t_0}\int_{s}^{t_0}c(r)c(s)drds\Big]+2\e\int_0^{t_0}c(s)ds\\
&=\Big[\int_0^{t_0}c(s)ds\Big]^2+2\e\int_0^{t_0}c(s)ds\\
&=\Big[\int_0^{t_0}c(s)ds+\e\Big]^2-\e^2
\end{align*}
which is a contradiction. 
\end{proof}

Using the continuity of $\dot\mu_a^\sharp$ together with~\eqref{e:intIneq} to apply Lemma~\ref{l:intIneq} with $b(t)=\sqrt{\dot \mu_a^\sharp(t)}$, $c(t)=\sqrt{\dot\alpha_a^\sharp(s)}$, we have
$$\sqrt{\dot \mu_a^\sharp(t)}\leq \int_0^t \sqrt{\dot \alpha_a^\sharp(s)}ds.$$
Now,
\begin{equation*}
\begin{aligned}
\mu(F([0,L]\times \Sigma)&=\int_0^L \dot\mu_1^\sharp(t)dt\leq \int_0^L\Big|\int_0^t\sqrt{\dot \alpha_1^\sharp(s)}ds\Big|^2dt\\
&=\left\|V\sqrt{\dot\alpha_1^\sharp}\right\|^2_{L^2([0,L])}\leq \|V\|_{L^2([0,L])\to L^2([0,L])}^2\int_0^L\dot\alpha_1^\sharp(t)dt
\end{aligned}
\end{equation*}
where $V:L^2([0,L])\to L^2([0,L])$ is the Volterra operator. 
Since the Volterra operator acting on $L^2([0,L])$ has norm $(2L)/\pi$
(\cite{Ha:67}, Problem 188) and
$$
\alpha(F([0,L]\times \Sigma))=\int_0^L\dot\alpha_1^\sharp(t)dt,
$$ 
we have
$$
\mu(F([0,L]\times \Sigma))\leq \frac{4L^2}{\pi^2}\alpha(F([0,L]\times \Sigma)).
$$
\end{proof}

\subsection{Completion of the proof of~\eqref{upperbound} in the boundaryless case}
Let $$\Sigma=\big\{\rho \in T^*_{\partial B(0,R)} M\colon (H_p
r)(\rho)<0 \big\},$$
where $r$ is the radial variable defined in the Euclidean end(s) as in Definition~\ref{def:euc}.
Thus, these are the inward-pointing covectors over the boundary of the
ball of radius $R.$  By convexity of Euclidean balls,
$$
 \bigcup_{t >0} \exp(t H_p) (\Sigma)\supset \mc{I}^c,\quad
  \bigcup_{t \leq 0} \exp(t H_p) (\Sigma)\subset \mc{I},
$$
and $F$ as in Proposition~\ref{p:volterra} is a diffeomorphism onto
its image.

Then, by~\eqref{transport1} and
Lemmas~\ref{lemma:CS},~\ref{l:incoming}, and~\ref{l:lowReg}, $(\mu,
\mu^j, \alpha,\Sigma)$ satisfy the hypotheses of
Proposition~\ref{p:volterra} with $\varphi_t=\exp(tH_p)$. Therefore,

$$
\mu\Big(\bigcup_{t=0}^L \varphi_t(\Sigma\cap S^*\!M)\Big)\leq \frac{4L^2}{\pi^2}\alpha(T^*\!M),
$$
where we may take $L=L(g,\rprime).$
So, since $0\leq \chi^2\leq 1$ and
$$
\supp \chi\cap S^*\!M\subset\bigcup_{t=0}^{L(g,\rprime)} \varphi_t(\Sigma\cap S^*\!M),
$$
we conclude
$$
\limsup_{h\to 0}\|\chi u\|_{L^2}=\sqrt{\mu(\chi^2)}\leq \frac{2L(g,\rprime)}{\pi}\sqrt{\alpha(T^*\!M)}\leq \liminf\frac{2L(g,\rprime)}{\pi}\|f(h)\|_{L^2}.
$$
This completes the proof of~\eqref{upperbound} in the case of a manifold without boundary.

\subsection{Uniformity of ${k}_0$ and Sobolev estimates}\label{sec:uniform}

{In this section we prove the uniformity of ${k}_0$ statement from
Theorem~\ref{theorem:main}, as well as the more general Sobolev
mapping property \eqref{Sobolevupperbound}.  We will omit the terms $L, L_j$ from this
discussion for brevity, but will on the other hand prove a slightly
more general
statement about metric perturbations as follows.}

We begin with uniformity in the case $s=0$, i.e., in the basic $L^2$
estimate.  Fix $R_1>0$ and let $\mc{C}$ be a subset of $\CR^{1,1}$
metrics so that for $g\in \mc{C}$, $g$ is nontrapping and
$\supp (I-g)\subset B(0,R_1)$. Furthermore, suppose that for any
$\{g_k\}_{k=1}^\infty\subset \mc{C}$, there exists a subsequence
$\{g_{k_m}\}_{m=1}^\infty$, $\gamma>0$ and $g\in \mc{C}$ so that
$\|g_{k_m} -g\|_{\CR^{1,\gamma}}\to 0$ and
$\lim_{m\to \infty} L(g_{k_m},R_1)\geq L(g,R_1).$ Let
$\chi\in \CR_c^{\infty}(B(0,R);[0,1])$ and $R_1<\rprime<\rpprime<R$ so
that $\supp \chi \subset B(0,\rprime)$.

 We start by showing that there exists ${k}_0>0$ so that for all $g\in \mc{C}$, and ${k} >{k}_0$, 
\begin{equation}
\label{e:unEst1}
\|\chi (P(g)-{k}^2-i0)^{-1}\chi\|_{L^2\to L^2}\leq \frac{2
  L(g,\rpprime)}{\pi k}
\end{equation}

Suppose not. Then since $L(g,\rprime)<L(g,\rpprime)$, there exists $\{g_k\}_{k=1}^\infty\subset \mc{C}$,
$h_k\to 0$, $f_k\in L^2(M)$, $u_k\in H^s_{\loc}(M)$, with
$\|\chi u_k\|_{L^2}=1$, and $\delta>0$
so that 
\begin{gather}
(h_k^2P(g_k)-1)u_k=h_k\chi f_k,\nonumber\\
1=\|\chi u_k\|_{L^2}\geq \frac{2 L(g_k,\rprime)\|f_k\|_{L^2(M)}}{\pi}+\delta.\label{e:contradictAssume}
\end{gather}
Extracting subsequences, we may assume that $g_k\to g$ in $\CR^{1,\gamma}$ for some $\gamma>0$ and $\lim L(g_k,\rprime)\geq L(g,\rprime).$ Note also that by Lemma~\ref{l:l2loc}, $u_k$ is uniformly bounded in $L^2_{\loc}$ and hence, by extracting further subsequences, we may assume that $u_k$ has defect measure $\mu$, $\chi f_k$ has defect measure $\alpha$, and $u_k,\chi f_k$ has joint defect measure $\mu^j$. Let $\Sigma$ denote the set of directly incoming points on $T^*_{\partial B(0,\rprime)}M$. By Lemmas~\ref{lemma:CS},~\ref{l:incoming}, and~\ref{l:lowReg}, $(\mu, \mu^j, \alpha,\Sigma)$ satisfy the hypotheses of Proposition~\ref{p:volterra} with $\varphi_t=\exp(tH_p)$ and $p=|\xi|^2_g-1$. Therefore, 
$$
\mu\Big(\bigcup_{t=0}^L \varphi_t(\Sigma\cap S^*\!M)\Big)\leq \frac{4L^2}{\pi^2}\alpha(T^*\!M),
$$
and we arrive at a contradiction just as above.

Finally, to obtain the bound \eqref{Sobolevupperbound}, we begin by
considering the case $s=2.$  We compute
\begin{multline*}
(P(g)+{k}^2)\chi (P(g) -{k}^2-i0)^{-1}\chi  f=\chi^2 f +([-\Delta_{\RR^n},\chi]+2{k}^2\chi)(P(g)-{k}^2-i0)^{-1}\chi f.
\end{multline*}
We next consider $[-\Delta_{\RR^n},\chi]$ and show that there exists $k_1>0$ so that for $k>k_1$ and $g\in  \mathcal{C}$, 
\begin{equation}
\label{e:boundReg}
\| [-\Delta_{\RR^n},\chi](P-k^2-i0)^{-1}\chi \|_{L^2\to L^2}\leq \sup_{S^*M}|H_{_{|\xi|^2}}\chi|\frac{2L(g_k,R'')}{\pi}.
\end{equation}
Suppose that~\eqref{e:boundReg} does not hold. Then there is $\delta>0$ and a sequence $h_k\to 0$, $f_k\in L^2(M)$, $u_k\in H^2(M)$ with $\|f_k\|_{L^2}=1$, $u_k$ so that 
\begin{equation*}
\begin{gathered}
(h_k^2P(g_k)-1)u_k=\chi h_kf_k,\quad u_k|_{\partial M}=0,\\
\| h_k^{-1}[-h_k^2\Delta_{\RR^n},\chi]u_k\|_{L^2}\geq \sup_{S^*M}|H_{_{|\xi|^2}}\chi|\frac{2(L(g_k,R')+\delta)}{\pi}
\end{gathered}
\end{equation*}
 Let $\tilde{\chi}\in \CR_c^\infty(B(0,R');[0,1])$ with $\tilde{\chi}\equiv 1$ on $\supp \chi$. 
As before, we may assume that $u_k$ has a defect measure and then, by Lemma~\ref{l:suppBoundary} applied with $\chi$ replaced by $\tilde{\chi}\in C_c^\infty(B(0,R'))$,
\begin{align*}
\sup_{S^*M}|H_{_{|\xi|^2}}\chi^2|\frac{4(L(g_k,R')+\delta)^2}{\pi^2}&\leq\| h^{-1}[-h^2\Delta_{\RR^n},\chi]\tilde{\chi}u_k\|_{L^2}^2\to \mu(|H_{_{|\xi|^2}}\chi|^2\tilde{\chi}^2)\\
&\leq \sup_{S^*M}|H_{_{|\xi|^2}}\chi|^2 \mu( \tilde{\chi}^2)\leq  \sup_{S^*M}|H_{_{|\xi|^2}}\chi|^2\frac{ 4(L(g,R'))^2}{\pi^2}.
\end{align*}
a contradiction.

Thus, by \eqref{e:unEst1}, for ${k}>{k}_0$, and any $g\in \mc{C}$,
\begin{align*}
&\|(P(g)+{k}^2)\chi(P(g)-{k}^2-i0)^{-1}\chi f\|_{L^2}\\
&\qquad\qquad\leq 2{k}^2\|\chi(P(g)-{k}^2-i0)^{-1}\chi f\|_{L^2}+\|\chi^2f\|+\|[-\Delta_{\RR^n},\chi](P(g)-{k}^2-i0)^{-1}\chi f\|\\
&\qquad\qquad\leq \frac{4L(g,\rpprime){k}}{\pi}\|\chi f\|_{L^2}+\|[-\Delta_{\RR^n},\chi](P(g)-{k}^2-i0)^{-1}\chi f\|_{L^2}+\|\chi^2f\|_{L^2}.
\end{align*}
Using~\eqref{e:boundReg}, we have for ${k}>{k}_1$, and any $g\in \mc{C}$,
$$
\|[-\Delta_{\RR^n},\chi](P(g)-{k}^2-i0)^{-1}\chi f\|_{L^2}\leq \sup_{S^*M}|H_{_{|\xi|^2}}\chi|\frac{2L(g_k,R'')}{\pi}\|f\|_{L^2}.
$$
In particular, letting 
$$
{k}>\max\Big({k}_0,{k}_1, \frac{\sup_{S^*M}|H_{_{|\xi|^2}}\chi|2L(g,\rpprime)+\pi}{ 4(L(g,R)-L(g,\rpprime))}\Big),
$$
we have for all $g\in \mc{C}$,
\begin{equation}
\label{e:unEst2}
\begin{aligned}
\|\chi(P(g)-{k}^2-i0)^{-1}\chi f\|_{H^2}&=\|(P(g)+{k}^2)\chi(P(g)-{k}^2-i0)^{-1}\chi f\|_{L^2}\\
&\leq \frac{4L(g,R){k}}{ \pi}\| f\|_{L^2}.
\end{aligned}
\end{equation}
Since $L(g,\rpprime)\leq L(g,R)$, the general case of $0\leq s\leq 2$ then follows by interpolation between~\eqref{e:unEst1} and~\eqref{e:unEst2}

If we choose $\mathcal{C}$ to be a subset of an open neighborhood in the
  $\CR^{2,\alpha}$ topology of a given
  $\CR^{2,\alpha}$ nontrapping metric, then provided this
  neighborhood is chosen sufficiently small, all the metrics in
  $\mathcal{C}$ are nontrapping, and the subsequence condition is
  guaranteed by the compactness of the embedding $\CR^{2,\alpha} \hookrightarrow
  \CR^{1,\gamma};$ indeed we have subsequential convergence in
  $\CR^{1,1},$ which ensures that $\lim L(g_{k_m}, R_1)
  =L(g,R_1).$  This choice of $\mathcal{C}$ then gives the uniformity
  assertion from Theorem~\ref{theorem:main}.
This completes the proof of~\eqref{upperbound} in the case of a manifold without boundary.

\section{Manifolds with boundary}

We now turn to the case of manifolds with boundary.  Our treatment of
the propagation of defect measures in this setting is motivated by
previous work of Miller \cite{Miller} and Burq--Lebeau
\cite{BurqLebeau:2001}; see also \cite{Bu:97}.

Let $\Mtilde$ be an open manifold extending $M$ and extend $P$ to an
operator on $\Mtilde$. We then let $\underline{u}$ be the extension of $u$
by $0.$
As in the boundaryless case, we will assume that we have a sequence
$h=h_j \downarrow 0$ with
$$Pu(h)=h\chi f(h)\text{ in }M,\qquad u|_{x_1=0}=0,$$
and we take
$$
\smallnorm{f(h)}=1, 
$$
$$
\lim_{h\to 0}\|\chi u\|_{L^2}=M>L(\Omega, g, R).
$$

Since propagation of singularities is a local consideration, we may employ
Riemannian normal coordinates $(x_1, x')$ in which $\pa M$ is given by
$x_1=0,$ $M=\{x_1>0\}$, and the operator $P$ is given by
\begin{equation}\label{operatornormalcoords}
P=(hD_{x_1})^2 + R(x_1,x', hD_{x'}) + h E
\end{equation}
for some self-adjoint $E \in \Diff^1_h$ and for $R \in \Diff^2_h$ with
symbol $r(x_1,x', \xi').$

We proceed as before, letting $\mu$ be a defect measure of $\underline{u},$ $\alpha$ be a defect measure of $\underline{\chi f},$
$\mu^j$ a joint defect measure of $\underline{u}$ and $\underline{\chi f}$.
Finally, let  $\dot{\nu}$ be the defect measures of
$hD_{x_1}u|_{x_1=0}$ on $\pa M.$  Thus,
\begin{equation}\label{defectmeasures2}
\begin{aligned}
\lim_{h\downarrow 0} \ang{a(x,hD) \underline{u}, \underline{u}} &= \int a \, d\mu\\
\lim_{h\downarrow 0} \ang{a(x,hD) \underline{\chi f}, \underline{\chi f}} &= \int a \, d\alpha\\
\lim_{h\downarrow 0} \ang{a(x,hD) \underline{\chi f}, \underline{u}} &=
\int a \, d\mu^j,\\
\lim_{h\downarrow 0} \ang{b(x',hD') hD_x u, h D_x u}_{\pa M} &=
\int_{\pa M} b \, d\dot{\nu}.
\end{aligned}
\end{equation}

\subsection{Local study of the measure $\mu$}

\begin{lemma}
For $a=a_0(x,\xi')+a_1(x,\xi')\xi_1$ with $a_i\in \CR_c^{\infty}([0,\e)\times T^*\partial M)$, we have
$$\mu(H_pa) =2\Im \mu^j(a)-\dot{\nu}( a_1|_{x_1=0}).$$
\end{lemma}
\begin{proof}
Here we work in a small neighborhood of the boundary and write 
$$A=a_0(x,hD')+a_1(x,hD')hD_{x_1}.$$
Equation~\eqref{basiccommutator} must now be modified
owing to boundary terms arising in integration by parts.  In
particular,
\begin{multline*}
\langle (hD_{x_1})^2Au,u\rangle_M\\
=ih(\langle hD_{x_1}Au,u\rangle_{L^2(\partial M)}+ \langle Au,hD_{x_1}u\rangle_{L^2(\partial M)})+\langle Au,(hD_{x_1})^2u\rangle_{L^2(M)}.
\end{multline*}
There is also a term
arising when we integrate by parts to change $\langle f,Au\rangle $ to $\langle Au,f\rangle$ but this term is $O(h)$. Now, 
\begin{align*}
\langle hD_{x_1}Au,u\rangle_{L^2(\partial M)}&= \langle a_0 hD_{x_1}u+a_1(hD_{x_1})^2u,u\rangle_{L^2(\partial M)}+o(1)\\
&=\langle a_0 hD_{x_1}u+a_1(1+h^2\Delta_{\partial M})u,u\rangle_{L^2(\partial M)}+\langle a_1h\chi f,u\rangle_{\partial M}+o(1)
\end{align*}
Now, with Dirichlet boundary conditions, this is actually just 0.
Moreover, 
$$\langle Au,hD_{x_1}u\rangle_{L^2(\partial M)}=\langle [a_0+a_1hD_{x_1}]u,hD_{x_1}u\rangle =\langle a_1hD_{x_1}u,hD_{x_1}u\rangle.$$
So~\eqref{basiccommutator} now reads 
$$ih^{-1}\langle [P,A]u,u\rangle_M=\Im \langle A\chi f,u\rangle_M -\langle a_1hD_{x_1}u,hD_{x_1}u\rangle_{\partial M}+o(1).\qed$$\noqed
\end{proof}

Now, we want to upgrade this to a statement for all functions in $\CR_c^{\infty}(T^*M).$ 
First, we need to show that $\mu$ is supported on $S^*\!M$. 
\begin{lemma}
\label{l:suppBoundary}
Let $\mu$ be a defect measure associated to $u$ as above.  Then $\supp \mu\subset S^*\!M$ and $\mu(\chi^2)=1$. Moreover, for $b\in S^1(T^*M)$, supported away from $\partial M$, 
$$
\lim_{h\to 0}\|Op_h(b)\chi u\|_{L^2}^2=\mu(|b|^2\chi^2).
$$
\end{lemma}
\begin{proof}
Note that 
$$
P\underline{u}=1_{x_1\geq 0}Pu+h(-\delta(x_1)\otimes h\partial_{x_1}u|_{x_1=0}+h\delta'(x_1)\otimes u|_{x_1=0})
$$
Now, suppose $a\in \CR_c^{\infty}(T^*\Mtilde)$ has $\supp a\cap \{p=0\}=\emptyset.$ Then, there exist $E\in \Psi^{-2}$ so that $a(x,hD)=EP+O_{\mathcal{D}'\to \CR^{\infty}}(h^\infty)$. Therefore, 
\begin{align*}
&a(x,hD)\underline{u}=a(x,hD)1_{x_1\geq 0}Pu +hE(-\delta(x_1)\otimes h\partial_{x_1}u|_{x_1=0})
\end{align*}
Now, since $E:H_h^s\to H_h^{s+2}$, is bounded and $\|h\partial _{x_1}u|_{x_1=0}\|\leq C$, we have
$$a(x,hD)\underline{u}=a(x,hD)1_{x_1\geq 0}Pu +O_{L^2}(h^{1/2}).$$
In particular, this implies that $\mu(a)=0$ completing the proof.

To see that $\mu(\chi^2)=1$, let $\psi\in \CR_c^{\infty}(\re)$ with $\psi\equiv 1$ on $|s|\leq 2$. Then, as above there exists $E\in S^{-2}$ so that 
$$\chi(1-\psi)(|hD|)\chi=EP+O_{\mc{D}'\to \mc{S}}(h^\infty)$$
and hence again
$$\chi(1-\psi)(|hD|))\chi\underline{u}=O_{L^2}(h^{\frac{1}{2}}).$$
In particular, 
\begin{align*}
1=\langle \chi^2\underline{u},\underline{u}\rangle&=\langle \chi\psi(|hD|)\chi \underline{u},\underline{u}\rangle +O(h^{\frac{1}{2}})\\
&=\int \chi^2(x)\psi(|\xi|)d\mu=\int \chi^2d\mu
\end{align*}
where in the last line we use the fact that $\mu$ is supported on $S^*\!M$.

The last claim follows the same arguments as in the boundaryless case.
\end{proof}

We now introduce notation for the even and odd parts of smooth
functions on $T^*\!M.$
\begin{definition}
  For $a \in \CI(T^*\!M),$ let
\begin{align*}
  a_o(x,\xi_1,\xi') &=\frac{
                      a(x,\xi_1,\xi')-a(x,-\xi_1,\xi')}{2\xi_1},\\
  a_e(x,\xi_1,\xi') &=\frac{ a(x,\xi_1,\xi')+a(x,-\xi_1,\xi')}{2}
\end{align*}
Thus, $a=a_e+\xi_1 a_o$ and $a_e,$ $a_o$ are both even functions of $\xi_1.$
\end{definition}

Next we upgrade the test functions for which we can compute $\mu(H_pa)$. 
\begin{lemma}
\label{l:diff}
For $a\in \CR_c^{\infty}(T^*\! \Mtilde)$, 
$$\mu(H_pa)=2\Im \mu^j(a)-\dot{\nu}( a_o|_{x_1=0})$$
In particular, for $a=a(x,x_1\xi_1,\xi')\in \CR_c^{\infty}(\Tbstar\overline{\re^n_+})$, 
$$\mu(H_pa)=2\Im \mu^j(a).$$
\end{lemma}
\begin{remark}
    Here, $\Tbstar \overline{\re^n_+}$ denotes the $b$-cotangent
    bundle of this local model for a manifold with boundary,
    constructed so that
    smooth functions on this space are simply functions of
    ($x,x_1\xi_1,\xi')$.
     See Section~\ref{s:tb} for a more precise description of this
     space and for additional references. We are regarding $a$ here
     as extended to a smooth function on $T^*\Mtilde;$ the choice of
     extension is immaterial, as $\mu$ is supported on $M.$
\end{remark}
\begin{proof}
  Since $a_o$ and $a_e$ are both even functions, we may write
  $$
a_{e/o}(x,\xi_1,\xi') =\ta_{e/o}(x,\xi_1^2, \xi')
$$
for smooth functions $\ta_{e/o};$
thus,
$$a= \ta_e(x,\xi_1^2,\xi')+\ta_o(x,\xi_1^2,\xi')\xi_1.$$
Now, since $H_p$ is tangent to $S^*\Mtilde$, $H_p1_{S^*\Mtilde}a=1_{S^*\Mtilde}H_pa.$ In particular, since $\supp \mu \subset S^*\Mtilde$, 
$$\mu(H_pa)=\mu(H_p1_{S^*\Mtilde}a).$$
Now, 
$$S^*\Mtilde:=\{\xi_1^2-r(x,\xi')=0\}.$$ 
Therefore, 
$$1_{S^*\Mtilde}a=1_{S^*\Mtilde}[\ta_e(x,r(x,\xi'),\xi')+\ta_o(x,r(x,\xi'),\xi')\xi_1].$$
Let $b_{e/o}=a_{e/o}(x,r(x,\xi'),\xi')$. Then we have
\begin{align*}
\mu(H_pa)&=\mu(H_p1_{S^*\Mtilde}[b_e+b_o\xi_1])\\
&=\mu(H_p(b_e+b_o\xi_1))\\
&=2\Im \mu^j(b_e+b_o\xi_1){-}\dot{\nu}(b_o|_{x_1=0}).
\end{align*}
Now, since $\supp \mu\subset S^*\!M$, $\supp \mu^j\subset S^*\!M$ (by
Lemma~\ref{lemma:CS}) and hence 
$$
\mu^j(b_e+b_o\xi_1)=\mu^j(a).
$$
Therefore, we have
$$\mu(H_pa)=2\Im\mu^j(a)-\dot{\nu}(b_o|_{x_1=0}).$$
\end{proof}

Before proceeding further, we recall the decomposition of $T_{\pa
  M}^*M$ into elliptic, hyperbolic, and glancing regions.   In the
notation of \eqref{operatornormalcoords}, with $r$ the symbol of $R(x_1,x',hD_{x'}),$
we write
\begin{align*}
  (y,\xi_1,\xi') &\in \E \text{ if } r(0,x',\xi') >0\\
  (y,\xi_1,\xi') &\in \hyp \text{ if } r(0,x',\xi') <0\\ 
    (y,\xi_1,\xi') &\in \gl \text{ if }r(0,x',\xi') =0 
  \end{align*}
  to denote the elliptic resp.\ hyperbolic resp.\ glancing sets.

  We further split the glancing set into subsets as follows: we let
\begin{align*}
  \gl_d &:=\gl \cap \{H_p^2 x_1 >0\},\\
    \gl_g &:=\gl \cap \{H_p^2 x_1 <0\}
\end{align*}
these are the \emph{diffractive} and \emph{gliding} sets respectively.
We will also employ the filtration of the remainder of the glancing
set given by the order of contact with the boundary as follows: 
$$
\gl^k:=\{ \rho \in \gl\mid (H_p^j x_1)(\rho)=0,\text{ for }0\leq j<k,\text{ and }(H_p^kx_1)(\rho)\neq 0\}.
$$ 
(We remark that this definition differs from the $G^k$ defined in
\cite[Section 24.3]{Hormander:v3} in that our $\gl^k$ does not include
higher-order glancing as well.)

With an open set $U \subset \Tbstar\!M$ fixed, we will also use the notation
$$
\tgl^k:= \{\rho\in U\mid \rho \text{ is connected to } \gl^k \text{
  by a generalized bicharacteristic}\}.
  $$
  We further let
  $$
\tgl^{\geq k} :=\bigcup_{k'\geq k} \tgl^{k'}.
  $$

\medskip
\needspace{2cm}
\noindent{\bf{Near Glancing:}}
\medskip

We now commence the study of the measure $\mu$ near the glancing set
$\gl.$  Here we follow an inductive strategy employed in \cite{BurqLebeau:2001}.

\begin{lemma}
\label{l:glancing1}
There exists a positive measure $\mu^\partial$ on the glancing set ${\gl}$ so that 
$$
\mu=\mu1_{x_1>0}+\delta(x_1)\otimes \delta(H_px_1)\otimes \mu^{\partial}.
$$
In particular, $\mu(\mathcal{H})=0.$
\end{lemma}
\begin{proof}
Let $a_\epsilon=\varphi(x,\xi')\epsilon \chi(\epsilon^{-1}x_1)$ with $\chi\in \CR_c^{\infty}(\mathbb{R})$, $\chi(0)=\chi'(0)=1$. 
 
By Lemma~\ref{l:diff},
$$\mu(H_pa_\epsilon)=2\Im \mu^j(a_\epsilon).$$
By the Dominated Convergence Theorem, 
$$2\Im \mu^j(a_\epsilon)\to 0$$
On the other hand,
$$H_pa_\epsilon=\chi'(\epsilon^{-1}x_1)\varphi(x,\xi')H_px_1+O(\epsilon).$$
So, by dominated convergence,
$$\mu(H_pa_\epsilon)\to \mu(1_{x_1=0}H_px_1\varphi).$$
In particular, 
$$\mu1_{x_1=0}(H_px_1\varphi)=0$$
for any $\varphi$ and hence $\mu1_{x_1=0}$ is supported on ${\gl}$. Since it is a positive measure, $\mu1_{x_1=0}$ takes the form specified.
\end{proof}

\begin{lemma}
\label{l:glancing2}
 We have
$$ -(H_p^2x_1)\mu^\partial=4\dot{\nu}1_{\gl}.$$
In particular, since $\mu^\partial$ and $\dot{\nu}$ are positive
measures,
$$\supp \mu^\partial \cup \supp \dot{\nu} \subset \{H_p^2x_1\leq0\},$$
hence
$$
\mu(\gl_d)=0,
$$
and for $E\subset \gl,$
$$\dot{\nu}(E)=\dot\nu \big( E\cap\{ H_p^2x_1<0\}\big).$$

\end{lemma}
\begin{proof}
Since $H_px_1=2\xi_1$, 
$$H_p(2a(x,\xi)\xi_1)=aH_p^2x_1+2\xi_1H_pa.$$
Let $a_\e =\chi(\e^{-1}x_1)\chi(\e^{-1}r(x,\xi'))2a(x,\xi')\xi_1$ where $\chi\in \CR_c^{\infty}(\re)$ has $\chi(0)=1$, $\chi'(0)=0$. 
Then
\begin{align*}
H_pa_\e&=\e^{-1}\chi'(\e^{-1}x_1)\chi(\e^{-1}r(x,\xi'))2a\xi_1H_px_1\\
&\quad+\e^{-1}\chi(\e^{-1}x_1)\chi'(\e^{-1}r(x,\xi')) 2a\xi_1H_pr\\
&\qquad +2\xi_1\chi(\e^{-1}x_1)\chi(\e^{-1}r(x,\xi'))H_pa+a\chi(\e^{-1}x_1)\chi(\e^{-1}r(x,\xi'))H_p^2x_1\\
&=a\chi(\e^{-1}x_1)\chi(\e^{-1}r(x,\xi'))H_p^2x_1+O(1)(|\chi'(\e^{-1}x_1)|+|\chi'(\e^{-1}r(x,\xi'))|+\e^{1/2})
\end{align*}
Here we have used that on $S^*\!M,$ $H_pr=-H_p\xi_1^2=O(\xi_1)$.
Then by dominated convergence,
$$
\mu(H_pa_\e)\to \frac{1}{2}\mu^\partial([H_p^2x_1]a) .
$$
On the other hand, since $a_\e=O(\sqrt{\e})$ on $S^*\!M,$ 
$$\mu(H_p a_\e)=2\Im \mu^j(a_\e)-\dot{\nu}(2 a(0, x', \xi')
\chi(\e^{-1} r))\to -\dot{\nu}(21_{{\gl}}a)$$
So, $-H_p^2x_1\mu^\partial =4\dot{\nu}1_{{\gl}}$ as claimed.
\end{proof}

\subsection{Invariance of the measure $\mu$}\hfill

\medskip
\needspace{2cm}
\noindent{\bf{The $b$-cotangent bundle, $\Tbstar M$}}
\medskip

\label{s:tb}
In our discussion of the invariance of $\mu$, it will be convenient to have some facts about the $b$-cotangent bundle, $\Tbstar M$. We refer the reader to~\cite[Section 18.3]{Hormander:v3} (where the notation $\tilde{T}^*M$ is used) and~\cite{Vasy:08} for more details on the $\Tbstar M$. Recall that $\Tbstar M$ is the dual to the vector bundle of vector fields tangent to $\partial M$. In local coordinates with $\partial M=\{x_1=0\}$, the bundle of vector fields tangent to $M$ is generated over $\CR^{\infty}(M)$ by the vector fields
$$
\{x_1\partial_{x_1},\partial_{x_2},\dots,\partial_{x_n}\}.
$$
Over $x_1>0$, $\Tbstar M$ is canonically diffeomorphic to $\TM$ with 
$$
(x,\zeta,\xi')\mapsto (x,\zeta/x_1,\xi').
$$
In fact, we can identify $\TM$ as a subset of $\Tbstar M$ using the canonical projection map $\pi:\TM\to \Tbstar M$ given in local coordinates by 
$$
\pi:(x,\xi_1,\xi')\mapsto (x,x_1\xi_1,\xi').
$$
Using this map, we can push forward the measure $\mu$ to a measure on $\Tbstar M$. It is actually this pushforward $\pi_*\mu$ for which we will show invariance under the generalized bicharacteristic flow.

Define
$$
H_p^G=H_p+\frac{H_p^2x_1}{H_{x_1}^2p}H_{x_1}.
$$
That is, $H_p^G$ is the gliding vector field. Then $H_p^G$ is tangent
to  $\gl$ (cf.\ \cite[Section 24.3]{Hormander:v3}).  
Let also $\varphi_t:\Tbstar M\to \Tbstar M$ denote the generalized bicharacteristic flow on $\Tbstar M$.

\medskip
\needspace{2cm}
\noindent{\bf{Invariance}}
\medskip

The aim of the rest of this section is to show that 
\begin{equation}
\label{e:diffBoundary}
\pi_*\mu(q\circ\varphi_t)-\pi_*\mu(q)=\int_0^t2\Im
\pi_*\mu^j(q\circ\varphi_s)ds,\qquad q\in \CR_c^{\infty}(\Tbstar M).
\end{equation}

The main lemma is as follows.
\begin{lemma}
\label{l:equiv}
Let $\mu$ and $\mu_1$ be measures on $T^*\!M$ supported by $S^*\!M$ with $\mu_1\ll\mu$. The following are equivalent:
\begin{enumerate}
\item \label{cond1}
\begin{equation}
\label{e:inv}
\pi_*\mu(q\circ\varphi_t)-\pi_*\mu(q)=\int_0^t\pi_*\mu_1(q\circ\varphi_s)ds,\qquad
q\in \CR_c^{\infty}(\Tbstar M). 
\end{equation}
\item \label{cond2}For $a=a(x_1,x',x_1\xi_1,\xi')$
\begin{equation}
\label{e:diff}
\mu(H_pa)=\mu_1(a).
\end{equation}
and $\mu({\gl_d}\cup{\mc{H}})=0$. 
\end{enumerate}
\end{lemma}

\noindent Once we prove this, Lemmas~\ref{l:diff}
and~\ref{l:glancing2} imply that $\mu$ satisfies condition~\eqref{cond2} with $\mu_1=2\Im\mu^j$ and hence satisfies~\eqref{e:diffBoundary}.
\medskip

\noindent{\bf{Proof of Lemma~\ref{l:equiv}}}\par
Here we follow \cite[Section 3.3]{BurqLebeau:2001}.

\
 
\noindent \textsc{\eqref{cond1} implies \eqref{cond2}:}

Let $a\in \CR_c^{\infty}(T^*\!M)$ with $a=a(x,x_1\xi_1,\xi')$. Then there exists $q\in \CR_c^{\infty}(\Tbstar M)$ such that $a=\pi^*q$.  Then, $q\circ\varphi_s(\rho)$ is differentiable from the left and right for all $s$ and 
\begin{equation}
\label{e:alongFlow}
\partial_s^\mp (q\circ\varphi_s)(\rho)=H_pa(\pi_{\pm}^{-1}(\varphi_s(\rho))),
\end{equation}
where $\pi_{\pm}^{-1}$ denotes the outward and inward pointing inverses of $\pi$. In particular, for $s$ such that $\varphi_s(\rho)\notin \mc{H}$, 
$$
\partial_s(q\circ\varphi_s)(\rho)=H_pa(\varphi_s(\rho)).
$$
 Now, $\pi_*\mu(\mc{H})=0$ since $\pi_*\mu$ satisfies~\eqref{e:inv}
 and $\mc{H}$ is transverse to the flow $\varphi_s$. Therefore, since
 $\pi_*\mu$ satisfies~\eqref{e:inv}, differentiating yields
 $$
 \pi_*\mu(\partial_t q\circ\varphi_t|_{t=s})= \pi_*\mu_1(q\circ\varphi_s)
 $$ 
 and setting $s=0$, gives
 $$\mu(H_pa)=\mu_1(a).$$
Finally, since $\gl_d$ is transverse to the flow, and $\pi_*\mu$
satisfies \eqref{cond1}, $\pi_*\mu(1_{\gl_d})=\mu(1_{\gl_d})=0$.  This
completes the proof of \eqref{cond2}.\qed 

\

\noindent \textsc{\eqref{cond2} implies \eqref{cond1}:}

\begin{lemma}\label{lemma:Hpbdecomp}
For $b\in \CR_c^{\infty}(T^*\!M)$, 
$$\mu(H_pb)=\mu_1(b)+\mu(1_{x_1=0}\partial_{x_1}rb_o|_{\xi_1^2=r(x,\xi')})+u^0(b_o|_{x_1=0,\xi_1^2=r(0,x',\xi')})$$
where $u^0$ is a measure supported $\mc{H}$,  $1_{x_1=0}\mu$ is supported by $\gl\setminus \gl_d$, and 
$$
b_o=\frac{b(x,\xi_1,\xi')-b(x,-\xi_1,\xi')}{2\xi_1}
$$
is the $\xi_1$ odd part of $b$. 
\end{lemma}
\begin{proof}
Write 
$$b=b_e+\xi_1b_o$$
where 
$$b_e=\frac{b(x,\xi_1,\xi')+b(x,-\xi_1,\xi')}{2}.$$
Then, $b_e=\tilde{b}_e(x,\xi_1^2,\xi')$, $b_0=\tilde{b}_o(x,\xi_1^2,\xi')$ and so on $S^*\!M$, 
$$
b_e=\tilde{b}_e(x,r(x,\xi'),\xi')=:a_e(x,\xi'),\qquad b_o=\tilde{b}_o(x,r(x,\xi'),\xi')=:a_o(x,\xi').
$$
So, since $H_p$ is tangent to $S^*\!M$, 
$$\mu(H_pb)=\mu(H_p(a_e+\xi_1a_o))=\mu(H_p\xi_1a_o)+\mu_1(a_e).$$
Now, let
\begin{equation}\label{eq:chi}
  \begin{aligned}
&\chi\in \CR_c^{\infty}(\re),\ \chi\equiv 1 \text{ near } 0,\ \chi'\leq 0
    \text{ on } x_1\geq 0,\\
    &\chi_\e(x_1)=\chi(\e^{-1}x_1).
\end{aligned}
    \end{equation}

Then, 
$$\mu(H_p(1-\chi_\e )\xi_1a_o))=\mu_1((1-\chi_\e) \xi_1a_o)\to \mu_1(1_{x_1>0}\xi_1a_o).$$
Now, 
$$\mu(H_p(\chi_\e\xi_1a_o))=\mu(\chi_\e\xi_1H_pa_o+\chi_\e \partial_{x_1}ra_o+2\e^{-1}\chi'_\e\xi_1^2a_o)=:I_\e+II_\e+III_\e.$$
By the dominated convergence theorem,
\begin{equation*}
I_\e \to \mu(1_{x_1=0}\xi_1H_pa_o),\qquad  II_\e \to \mu(1_{x_1=0}\partial_{x_1}ra_o).
\end{equation*}
Note that since $\mu(\mc{H})=0$ and $\xi_1=0$ on $\{x_1=0\}\setminus \mc{H}$, we have 
\begin{equation}
\label{e:part1}
I_\e \to 0,\qquad  II_\e \to \mu(1_{x_1=0}\partial_{x_1}ra_o).
\end{equation}

Now, observe that 
$$\mu_1((1-\chi_\e) \xi_1a_o)=\mu(H_p(1-\chi_\e)\xi_1 a_o)=\mu((1-\chi_\e)H_p\xi_1a_o)-III_\e.$$
So, 
$$III_\e=\mu((1-\chi_\e)H_p\xi_1a_o)-\mu_1((1-\chi_\e) \xi_1a_o)$$
and in particular, 
\begin{equation}
\label{e:part2} 
\lim_{\e \to 0}III_\e(a_o)=\mu(1_{x_1>0}H_p\xi_1a_o)-\mu_1(1_{x_1>0}\xi_1a_o).
\end{equation}
On the other hand for $a_o\geq 0$, 
$$III_\e(a_o)\leq 0.$$
Therefore, $III_\e$ is a family of measures with a weak limit. In particular, $III_\e\to u^0$ for some measure $u_0$ supported in $x_1=0$. In fact, this also shows that 
\begin{equation}
\label{e:u0def}
a_o\mapsto \mu(1_{x_1>0}H_p(\xi_1a_o))-\mu_1(\xi_1a_o)=:u^0(a_o).
\end{equation}
is a measure. We now check that $u^0$ is supported in $\mc{H}$. Note that once we show this the proof will be complete, since by~\eqref{e:part1} and~\eqref{e:part2} together with $\mu_1(\mc{H})=0$, we then have 
\begin{align*}
\mu(H_pb)&=\mu_1(a_e)+\mu_1(\xi_1a_o)+\mu(1_{x_1=0}\partial_{x_1}ra_o)+u^0(a_o)\\
&=\mu_1(b)+\mu(1_{x_1=0}\partial_{x_1}ra_o)+u^0(a_o)
\end{align*}
as desired.

Let $\chi,$ $\chi_\e$ be as in \eqref{eq:chi}. Note that for each $\e>0$,
$a_o\chi_\e(x_1)\chi_{\sqrt{\e}}(\xi_1)$ can be written as a smooth
function of $(x_1,x',x_1\xi_1,\xi')$. Using~\eqref{e:u0def}, we now decompose
$$
u^0(a_o\chi_\e(x_1)\chi_{\sqrt{\e}}(\xi_1))=u^0_1(\e)+u^0_2(\e)+u^0_3(\e)+u^0_4(\e)
$$
where
\begin{align*}
u^0_1(\e)&=\mu(1_{x_1>0}\chi_\e(x_1)\chi_{\sqrt{\e}}(\xi_1)\xi_1 H_pa_o),\\
  u^0_2(\e)&=\mu(1_{x_1>0}\e^{-1}\chi_\e'(x_1)\chi_{\sqrt{\e}}(\xi_1)2\xi_1^2a_o),\\
  u^0_3(\e)&=\mu(1_{x_1>0}\chi_\e(x_1)(\xi_1\e^{-\frac{1}{2}}
            \chi'_{\sqrt{\e}}(\xi_1)+ \chi_{\sqrt{\e}}(\xi_1))\partial_{x_1}ra_o),\\
  u^0_4(\e)&=-\mu_1(\xi_1a_o\chi_\e(x_1)\chi_{\sqrt{\e}}(\xi_1)).
\end{align*}
Clearly $u^0_1(\e) \to 0$ by dominated convergence.
The same is true of $u^0_2(\e):$ the function
$\chi_\e'(x_1)\chi_{\sqrt{\e}}(\xi_1)2\xi_1^2a_o$ is uniformly
bounded as $\e \downarrow 0$ since $\xi_1^2 \leq C \ep$ on $\supp \chi_{\sqrt{\e}}(\xi_1).$
Next,
$$u^0_3(\e) \to 0$$
as $ \xi_1\e^{-\frac{1}{2}} \chi'_{\sqrt{\e}}(\xi_1)+ \chi_{\sqrt{\e}}(\xi_1)$ likewise
remains uniformly bounded.
Finally, 
$$u^0_4(\e) \to -\mu_1(1_{x_1=0}1_{\xi_1=0}\xi_1a_o)=0.$$ 
In particular, this implies by the dominated convergence theorem that
$$u^0(a_o1_{x_1=\xi_1=0})=0$$
and so $u^0$ is supported by $\mc{H}$.
\end{proof}

\medskip

\noindent {\bf{Invariance away from glancing}}\par
We now prove the integral invariance statement \eqref{cond1} away from glancing.

\begin{lemma}
\label{l:hypInvariant}
Suppose
$\tilde{a}=\tilde{a}(x,\sigma,\xi')\in \CR_c^{\infty}$ and let
$a=\tilde{a}(x,x_1\xi_1,\xi')$. Furthermore, assume that for
$\rho \in \supp a$, on $[-T,0]$ $\varphi_{-t}(\rho)$ intersects
$x_1=0$ at most once and in a hyperbolic point. Then,
$$
\mu(a)-\mu(a\circ \varphi_{-T})=\int_{-T}^0\mu_1(a\circ\varphi_{-s})ds.
$$
\end{lemma}
\begin{proof}
Let $-T(\rho)$ be the time that $\varphi_{-T(\rho)}(\rho)\in \mathcal{H}$. Then
\begin{align*}
\int_{-T}^0H_p(a\circ\varphi_{s})(\rho)ds&=\int_{-T}^{-T(\rho)-0}H_p(a\circ \varphi_s)ds+\int_{-T(\rho)+0}^{0}H_p(a\circ\varphi_{s})ds\\
&=a\circ\varphi_{-T(\rho)-0}(\rho)-a\circ\varphi_{-T}(\rho)+a-a\circ\varphi_{-T(\rho)+0}(\rho)
\end{align*}
Now, since at $x_1=0$, $a(0,x',\xi_1,\xi')=a(0,x',-\xi_1,\xi')$, $a\circ\varphi_{-T(\rho)-0}=a\circ\varphi_{-T(\rho)+0}$ and we have 
$$\int_{-T}^0H_p(a\circ\varphi_{s})(\rho)ds=a(\rho)-a\circ\varphi_{-T}(\rho).$$

Now,  by Lemma~\ref{lemma:Hpbdecomp}
\begin{align*}
\int \int_{-T}^0H_p(a\circ \varphi_s)(\rho)dsd\mu(\rho)&=(H_p^*\mu)\Big(\int_{-T}^0a\circ \varphi_{s}ds\Big)\\
&=\mu_1\Big(\int_{-T}^0a\circ\varphi_{s}ds\Big)+u^0\Big[\Big(\int_{-T}^0 a\circ\varphi_sds|_{\mathcal{H}}\Big)_o\Big].
\end{align*}

Note that
$$ \rho\mapsto \int_{-T}^0 a\circ\varphi_sds|_{\mathcal{H}}(\rho)$$
is even under $\xi_1\mapsto -\xi_1,$ since points $(x=0,
y,\pm\xi,\eta)$ over the hyperbolic set are
both mapped to the same point in $T^*\!M^\circ$ under the short-time flow; thus the $u^0$ term vanishes. 
Therefore, 
$$\int_{-T}^0\int a\circ\varphi_{-s}(\rho)d \mu_1ds= \mu_1 \Big(\int_{-T}^0a\circ\varphi_{s}ds\Big)=\int a-a\circ\varphi_{-T}d\mu.$$
\end{proof}

\noindent {\bf{Invariance on the glancing set}}\par

We now turn to analysis at the glancing set.  Our strategy will be to
prove \eqref{cond1} for all $t<t_0$ fixed and small.  To do this, we
may break up the support of $q$ by a partition of unity and
work locally, assuming that $q$ is replaced by a function $a$ supported in a
neighborhood $U$ of a point $\rho\in \gl$; we will show that for each
such $U,$ if $t_0$ is taken to be small, \eqref{cond1} holds with $q$
replaced by $a.$

We will do this using a `layer stripping' argument. First, we `strip away' the hyperbolic layer using the invariance established in
Lemma~\ref{l:hypInvariant}. 
In particular, recall that $\tgl$ denotes the set
of generalized bicharacteristics in $U$ which encounter $\gl$. Then we
have already seen that~\eqref{cond1} holds on $\tgl^c$. Since
$\tgl^c$ is open, $\pi_*1_{\tgl^c} \mu $
satisfies
\begin{equation}
\label{e:inv2}
\pi_*1_{\tgl^c}\mu(q\circ\varphi_t)-\pi_*1_{\tgl^c}\mu(q)=\int_0^t\pi_*2\Im 1_{\tgl^c} \mu_1(q\circ\varphi_s)ds,\qquad q\in \CR_c^{\infty}(\Tbstar M). 
\end{equation}
In particular, since
\eqref{cond1} implies \eqref{cond2}, $1_{\tgl^c}\mu$
satisfies
\begin{equation}
\label{e:diff2}
1_{\tgl^c} \mu(H_pa)=1_{\tgl^c}\mu_1(a).
\end{equation}
 Subtracting
this from~\eqref{e:diff} for $\mu$, $\mu1_{\tgl}$
satisfies
\begin{equation}
\label{e:diff3}
\mu 1_{\tgl} (H_pa)=1_{\tgl}\mu_1 (a).
\end{equation} 
Thus we may turn our attention to studying the new measures $\mu
1_{\tgl},$ $1_{\tgl}\mu_1.$  By a slight abuse of notation, then
\emph{we assume henceforth that $\mu$ and $\mu_1$ are supported on $\tgl$.}

As we proceed through the proof, we will `strip away' the higher order
tangent layers. We start by showing that at points where the
bicharacteristics of $H_p$ are tangent to exactly order 2, the measure
is invariant in the sense of~\eqref{cond1}. Once we have shown this,
we can argue as we did to remove the hyperbolic set and replace the
measure $\mu$ by $\mu1_{\tgl^{\geq 3}}.$  We then prove invariance on $\tgl^k$ by induction.  As before,
the fact that~\eqref{cond1} implies~\eqref{cond2} allows us to assume
that $\mu$ is supported on $\tgl^{\geq k}$. Then, since the flow is
transverse to $\gl^k$ we may work with a measure supported on
$\tgl^k$ when showing invariance at order $k$. We will obtain this
invariance by using the transversality of the flow to $\gl^k$
to show that there is no jump in the measure across these points. In
particular, this argument will show that $\mu$ accumulates no mass across
$\gl^k$ and thus that the invariance continues to hold.

\medskip

\medskip

\noindent {\bf{The case of $\gl_d$}}\par
We choose the neighborhood of $U$ of $\rho \in \tgl_d$ so small that $U\cap (\mc{H}\cup \gl\setminus \gl_d)=\emptyset.$
 Next, suppose that $a\in \CR_c^{\infty}(U)$. Recall that
$$
\mu(H_pa)=\mu_1(a)+\mu(1_{x_1=0}\partial_{x_1}ra_o|_{\xi_1^2=r(x,\xi')})+u^0(a_o),
$$
with $u^0$ supported in $\mc{H}\cap\supp \mu$ and $\mu(\gl_d)=0$. Since $\supp \mu\subset \tgl$ and $\supp a\subset U$, we have 
$$
\mu(H_pa)=\mu_1(a).
$$
and hence~\eqref{e:inv} holds since the generalized bicharacteristics through $\gl_d$ are bicharacteristics of $H_p$. 

\medskip

\noindent{\bf{The case of $\gl_g$}}

Let $\tgl_g$ be the set of generalized bicharacteristics encountering
$\gl_g$ in $U$. Recall that we may assume
$\mu$ and $\mu_1$ are supported on $\tgl_g;$ may further choose $U$ small
enough so that $U \cap \tgl_g\subset \gl_g,$ hence we may in fact
assume at this stage that $\mu$ and $\mu_1$ are supported by $\gl_g.$

Let $a\in \CR_c^{\infty}(U)$. For $q=q(x_1,x',x_1\xi_1,\xi')$, $1_\gl H_p^Gq=1_\gl H_pq$. Moreover, since $H_p^Gx_1=H_p^G\xi_1=0$, letting $a_0=a|_{\xi_1=0}$,
\begin{equation}\label{glancingderivative}
1_\gl H_p^Ga=1_\gl H_p^Ga_0=1_\gl H_pa_0.
\end{equation}
Now, since on $U$, $\supp \mu\subset \gl_g$
\begin{equation}
\label{e:glide2}
\mu(H_p^Ga)=\mu(1_{\gl_g}H_p^Ga)=\mu(1_{\gl_g}H_p^Ga_0)=\mu(H_pa_0)=\mu_1(a_0)=\mu_1(1_{\gl_g}a_0)=\mu_1(a)
\end{equation}
Since $H_p^G$ generates the bicharacteristic flow in $\gl_g$, this implies~\eqref{e:inv}. 

\begin{remark}
  While it may seem at first that the argument used for $\gl_g$
  applies directly to higher order tangencies, we observe that it does
  not. In fact, the equality $\mu(H_pq)=\mu_1(q)$ for
  $q=q(x,x_1\xi_1,\xi')$ a priori only implies that $\mu$ mass can
  either `stick' to the boundary as in $\gl_g$ or instantly detach
  from the boundary as in $\gl_d$. It does not rule out either
  case. In $\gl_g$, the fact that trajectories which detach from the
  boundary instantly leave $\overline{M},$ while in $\gl_d$, the fact
  that $\mu(\gl_d)=0$, rule out detaching from and sticking to the
  boundary at $\gl_g$ and $\gl_d$ respectively. In order to determine
  whether mass at a point of higher order tangency sticks or detaches,
  we will use the transversality of the flow to such points to show
  that mass can neither accumulate nor dissipate.
\end{remark}

\medskip

\noindent{\bf{Higher order tangencies}}

Recall that 
$$
\gl^k:=\{ \rho \in \gl\mid (H_p^j x_1)(\rho)=0,\text{ for }0\leq j<k\text{ and }(H_p^kx_1)(\rho)\neq 0\}.
$$
We have shown that~\eqref{e:diffBoundary} holds for $a\in \CR_c^{\infty}(T^*\!M)$ with $\supp a\cap \gl^k=\emptyset$ for $k\geq 3$. We now proceed by induction on $k$. 

Let $r_0(x',\xi')=r(0,x',\xi')$ and
$r_1(x,',\xi')=\partial_{x_1}r(0,x',\xi')$. Then note that by
\cite[Lemma 24.3.1]{Hormander:v3},
$$
\gl^k=\{ x_1=0=1-r(0,x',\xi'),\, H_{r_0}^jr_1=0,\,0\leq j<k-2,\,H_{r_0}^{k-2}r_1\neq 0\}.
$$

For each $k,$ let $\tgl^k$
denote the set of generalized
bicharacteristics encountering $\gl^k$ in a small fixed neighborhood
of $\gl^k.$

Suppose~\eqref{e:inv} holds for $q\in \CR_c^{\infty}(T^*\!M)$ with
$\supp q\cap \tgl^k=\emptyset$ for $k\geq M-1$. We show
that~\eqref{e:inv} holds for $a\in \CR_c^{\infty}(T^*\!M)$ with
$a\cap \tgl^k=\emptyset$ for $k\geq M$.

Let $\rho \in \gl^{M}$ and $U$ a neighborhood of $\rho$ so that
$U\cap \gl^k=\emptyset$ for $k>M$. As before, we may assume $\mu$ and
$\mu_1$ are supported on $\tgl^{\geq M}$;  thus by our choice of $U$,
we may in fact assume without loss of generality that $\mu$,
$\mu_1$ are supported on $\tgl^M.$ Let $(x_2,\rho)$ be
coordinates on $T^*(\partial M)$ so that $\partial_{x_2}=H_{r_0}.$ Shrinking $U$
if necessary, by the implicit function theorem there
exists $\Theta(\rho)$ such that on $U,$
$$
\tgl^M\cap \gl^M=\{(x,\xi)\in \tgl^M\mid x_2=\Theta(\rho)\}.
$$ 
Let 
\begin{gather*}
\tgl_0:=\tgl^M\cap \gl^M,\quad \tgl^+:=\{(x,\xi)\in \tgl^M\mid  x_2>\Theta(\rho)\},\\
\tgl^-:\{(x,\xi)\in \tgl^M \mid x_2<\Theta(\rho)\}.
\end{gather*}

We write
\begin{gather*}
\mu1_{\tgl^M}=\mu_-+\mu_++\Upsilon,\\
\mu_-=1_{\tgl^-}\mu,\qquad \mu_+=1_{\tgl^+}{\mu},\qquad \Upsilon=1_{\tgl_0}\mu.
\end{gather*}
Then, for $q=q(x_1,x',x_1\xi_1,\xi')$ supported in $U,$
\begin{equation}\label{jump}
    \Upsilon(H_pq)=-\mu_-(H_pq)-\mu_+(H_pq)+\mu_1(q).\end{equation}
    Let 
$$
\widetilde{H}(\rho)=\begin{cases}
H_p(\rho)& \rho\notin \gl_g\\
H_p^G(\rho)&\rho \in \gl_g;
\end{cases}
$$
this vector field generates the bicharacteristic flow on $\tgl^M,$ so
by our inductive
hypothesis, \begin{equation}\label{flowderivmu}\widetilde{H}^*
  \mu=\mu_1\end{equation} on $\tgl^M \backslash \gl^M.$  Moreover, since on $\tgl^0$ we have $H_p^M x_1 \neq 0,$
the flow within $U$ (perhaps after further shrinking that set) for
forward and backward time is either in $\{x_1>0\}$ or lies entirely in
$\gl_g$ (with separate cases for forward/backward flow depending on
the parity of $M$ and the sign of $H_p^M x_1$).  In either case, the
flow stays away from $\hyp,$ hence is generated by $\widetilde{H}$ on
$\tgl^M \cap U.$

Now as in Lemma~\ref{lemma:AC} and the surrounding
discussion, given any continuous $a$ defined on $\gl_0$ and $\psi\in
\CR_c^\infty(\RR),$ we extend $a$ to be constant on the flow, and set
$$
\mu^\sharp_a(\psi)=\mu(\psi(s) a(\rho));
$$
here we are using $\rho\in \gl_0$ together with $s\in \RR$ as local
(non-smooth) coordinates on $\tgl^M$ defined by the flow, mapping
$(-\delta,\delta)\times \tgl_0$ homeomorphically to $\tgl^M$ via
$(s,\rho) \mapsto \exp(t\widetilde{H})(\rho).$  Just as in
Lemma~\ref{lemma:AC}, the flow invariance established on $\tgl^M
\backslash \gl_0$ implies that $\mu_a^\sharp$ is AC
with respect to Lebesgue measure on $\RR\setminus 0;$ moreover the Radon--Nikodym
derivative is itself bounded by a multiple of $\sup \lvert a\rvert,$ hence is a
function of $s$ with values in measures on $\gl_0,$ denoted $G(s).$
Finally, let $\tr\dot{\mu}_\pm=G(0^\pm)$
denote the restrictions of these functions to $\gl_0$ from left and
right; these are well defined since the relation \eqref{flowderivmu} together with the arguments in Lemma~\ref{l:cts}
show that $G$ is continuous  on $s\leq 0$ and on $s\geq 0$ separately but that there may be a jump at $s=0$.

Moreover, denoting the function $\psi(s)a(\rho)$ by $\psi\otimes a$, we have for $0\notin\supp \psi$,
$\mu(\psi'\otimes a)=\mu_1(\psi\otimes a).$ So in particular, 
$$\int [G(a)](s)\psi'(s)ds=\mu_1(\psi\otimes a).$$

Now, let $\chi\in \CR_c^{\infty}(\RR)$ with $\chi \equiv 1$ on $[-1,1]$ and $\supp \chi \subset (-2,2)$ and $\chi_\e(s)=\chi(\e^{-1}s)$. Then, 
\begin{align*}
\mu((1-\chi_\e(s))\tilde{H}(\psi \otimes a))&=\mu(\tilde{H}[(1-\chi_\e)\psi]\otimes a)+\mu([\psi \otimes a]\tilde{H}\chi_\e)\\
&=\mu(\tilde{H}[(1-\chi_\e)\psi\otimes a])+\mu(\e^{-1}\psi\chi'(\e^{-1}\cdot) \otimes a)\\
&=\mu_1((1-\chi_\e)\psi \otimes a)+\mu(\e^{-1}\psi \chi'(\e^{-1}\cdot)\otimes a)
\end{align*}
Then, 
\begin{align*}
\mu(\e^{-1}\psi\chi'(\e^{-1}\cdot) \otimes
  a)&=\e^{-1}\int_{\RR\setminus 0} [G(a)](s)\chi'(\e^{-1}s )\psi(s)ds\\
&= \psi(0)(G_a(0^-)-G_a(0^+))-\int_{\RR\setminus 0} G_a(s)\chi_\e(s)\psi'(s)ds+\mu_1(\chi_\e\psi\otimes a)\\
&\to \psi(0)(G_a(0^-)-G_a(0^+))+\mu_1((\psi \otimes a)1_{\tgl^0})
\end{align*}

In particular, 
\begin{equation}
\label{e:jumpjump}
\begin{aligned}
\mu_+(\tilde{H}\psi\otimes a)+\mu_-(\tilde{H}\psi\otimes a)&= \psi(0)([G(a)](0^-)-[G(a)](0^+))+\mu_1(\psi \otimes a)\\
&=\tr\dot{\mu}_+(\psi \otimes a)-\tr\dot{\mu}_-(\psi \otimes a)+\mu_1(\psi \otimes a).
\end{aligned}
\end{equation}

Now we claim that we may also apply \eqref{jump} to the function
$q=\psi\otimes a.$  We begin by approximating $a$ with a smooth
function on $\gl_0.$  Since $\widetilde{H}$ locally generates the flow
and is at least Lipschitz across $\gl_0,$ we find that the flow on
$\tgl^M$ is at least $\CR^1.$ We note further that
$\tgl^M$ is without loss of generality disjoint from $\hyp$ (after
appropriately shrinking $U$) since
depending on the parity and sign of $H_p^M x_1$ along $\gl_0$ the
flow is either gliding or enters $x_1>0$ in each of $\tgl^\pm.$  In
particular, $\pi(\tgl^M)\subset \Tbstar M$ is $\CR^1$. Moreover, 
$$
(-\delta,\delta)\times \gl_0\ni(s,q)\mapsto \varphi_s(q)\in \pi(\tgl^M)
$$
are $\CR^1$ coordinates on $\tgl^M$.  Hence
$b=\pi^*\psi\otimes a$ is $\CR^1$ on $\pi(\tgl^M)$ and we may extend $b$
to $\tilde{b}$, a $\CR^1$ function on $\Tbstar M$. Equivalently, setting $\tilde{q}=\pi^*\tilde{b}$, we have
$\tilde{q}=\tilde{b}(x,x_1\xi_1,\xi')$ with $\tilde{b}$ a $\CR^1$
function of its arguments and $\tilde{q}|_{\tgl^M}=\psi\otimes a.$

Since $\tilde{b}$ is $\CR^1$, we may approximate it in $\CR^1$ by smooth
functions $\tilde{b}_\e$ and hence we obtain using the dominated
convergence theorem that~\eqref{jump} holds for $\tilde{q}$. Finally,
since $\tilde{q}|_{\tgl^M}=\psi\otimes a$ and
$\supp \mu\subset \tgl^M$, we obtain~\eqref{jump} for
$q=\psi\otimes a$. Moreover, since $\mu_{\pm}$ are each supported in the
interior or $\gl_g,$ we can replace $\tilde{H}$ by $H_p$
in~\eqref{e:jumpjump}.

Now \eqref{jump} and \eqref{e:jumpjump} yield:
\begin{equation}
\label{e:Ups}
\Upsilon(H_p\psi\otimes a)=-\tr(\dot{\mu}_+)(\psi\otimes a)+\tr(\dot{\mu}_-)(\psi \otimes a)
\end{equation}
and hence 
$$
|\Upsilon(H_p\psi\otimes  a)|\leq C|\sup \psi \otimes a|.
$$

Since $\Upsilon$ is supported at $\tgl^0$, and the collection of functions $\psi\otimes a$ is dense, this implies that for all $q$,
$$
|\Upsilon(H_pq)|\leq C|\sup q|.
$$
Then, using that $\Upsilon$ is a measure supported at $\tgl_0$ and $H_p$ is
transverse to $\tgl_0,$ this implies $\Upsilon=0$. Finally, inserting
$\Upsilon=0$ into~\eqref{e:Ups} yields
$$
\tr(\dot{\mu}_-)=\tr(\dot{\mu}_+).
$$
Moreover, since $\mu_1$ is absolutely continuous with respect to $\mu$, $\mu_11_{\tgl^0}=0$. 

Let $\psi_{t_0}(s)=\psi(s+t_0)$. Then $(\psi\otimes a)\circ\varphi_t=\psi_t\otimes a$ and
\begin{align*}
\mu(\psi_{t_0}\otimes a)-\mu(\psi\otimes a)&=\int_{0}^{t_0}\partial_t \int [G(a)](s)\psi(s+t)dsdt\\
&=\int_0^{t_0}\int_{-\infty}^0 [G(a)](s)\psi'(s+t)ds+\int_0^\infty [G(a)](s)\psi'(s+t)dt\\
&=\int_0^{t_0} \int_{-\infty}^\infty \mu_1(\psi_t\otimes a)ds+ \tr\dot{\mu}_-(a)\psi(t)-\tr\dot{\mu}_+(a)\psi(t)dt\\
&=\int_{0}^{t_0}\mu_1(\psi_t\otimes a)dt
\end{align*}
In particular, we have 
$$
\mu((\psi\otimes a)\circ\varphi_{t_0})-\mu(\psi \otimes a)=\int_{0}^{t_0}\mu_1( [\psi \otimes a]\circ\varphi_t)dt.
$$
Since the collection of functions of the form $\psi\otimes a$ is dense in $\CR^0(\Tbstar M)$, this completes the proof that $\mu$ is invariant on $\tgl^M$.

As there are no infinite order tangencies, $\gl=\bigcup_M\gl^M$ and this completes the proof of~\eqref{e:inv}.\qed

Finally, to complete the proof of~\eqref{upperbound} we use Proposition~\ref{p:volterra} and the following lemma.
\begin{lemma}
Suppose that $\mu$ satisfies~\eqref{e:diffBoundary}. Let $\psi \in
\CR_c^1(\re)$, $\Sigma$ be a hypersurface transverse to the generalized
bicharacteristic flow so that $\varphi_t:\re\times \Sigma\to \Tbstar M$
is a homeomorphism onto its image, and\ $a\in \CR_c^0(\Sigma)$. Define
$f_{\psi,a}(\varphi_t(\rho))=\psi(t)a(\rho).$   Then
$$
\pi_*\mu(f_{\partial_t\psi,a})=2\Im \mu^j(f_{\psi,a}).
$$
\end{lemma}
\begin{proof}
We have for $q\in \CR_c^{\infty}(\Tbstar M)$, 
$$
\partial_t\pi_*\mu(q\circ\varphi_t)|_{t=0}=2\Im \mu^j(q).
$$
Therefore, it is enough to show that we can move the derivative inside $\pi_*\mu$. This follows from the dominated convergence theorem since $\pi_*\mu(\mc{H})=0$. 
Now approximation of $f_{\psi,a}$ by smooth functions gives the result. 
\end{proof}

In order to obtain the required uniformity, we adjust the set $\mc{C}$ by fixing neighborhood $U$ of $\Omega$ in which we assume that for $\{g_k\}_{k=1}^\infty\subset \mc{C}$, we not only have $\|g-g_{k_m}\|_{\CR^{1,\gamma}}\to 0$ for some $\gamma>0$ but also, $\|g-g_{k_m}\|_{\CR^{\infty}(U)}\to 0$. Arguing as in Section~\ref{sec:uniform} then completes the proof ot Theorem~\ref{thm:2} when $\Omega\neq \emptyset$.

\section{Lower bounds on manifolds without boundary}

The idea behind our proof of the lower bounds in Theorem~\ref{theorem:main} is that near any segment of a geodesic $\gamma$, which is not trapped, $P$ is \emph{globally} microlocally equivalent to $hD_{x_1}$ and hence it is enough to construct examples which saturate the $L^2_{\comp}\to L_{\loc}^2$ bounds for the solution operator for $hD_{x_1}$ where we impose the condition that the solution be supported in $x_1\geq 0$.

Let $U$ be an open set.  Assume there exists a null bicharacteristic curve
$\gamma$ such that 
\begin{equation}
\label{e:U}
\{\pi(\gamma(s)),\ s \in (0, L)\} \subset U,\quad \pi(\gamma(s)) \notin U \text{
  for } s \notin (0,L).
\end{equation}
For any interval $I$ (open or closed), let
$$
\Gamma_I\equiv \{\gamma(s),\ s \in I\}\subset T^*\!M,\qquad \Gamma^0_I=\{x_1\in I,\xi_1=0,x'=\xi'=0\}\subset T^*\RR^n
$$

\begin{lemma}\label{lemma:longdarboux}
There exist neighborhoods $V\subset T^*\!M$ of $\Gamma_{[0,L]}$, $U\subset T^*\RR^n$ of $\Gamma^0_{[0,L]}$ and a symplectomorphism $\kappa:U\to V$ with
$$
\kappa: \Gamma^0_{[0,L]} \to \Gamma_{[0,L]},\quad \kappa^* p=\xi_1.
$$
\end{lemma}
\begin{proof}
Let $\rho_0=\gamma(0)$. Then by Darboux's theorem, there exist neighborhood $V_1\subset T^*\!M$ of $\rho_0$, $U_1\subset T^*\RR^n$ of $0$ and a symplectomorphism $\kappa_1:U_1\to V_1$ so that 
$$
\kappa_1^*p=\xi_1,\qquad \kappa_1:\Gamma^0_{[0,L]}\cap U_1\to \Gamma_{[0,L]}\cap V_1. 
$$

We will extend $\kappa_1$ so that its image covers a neighborhood of $\Gamma_{[0,L]}$. For this, let $\Sigma:=\kappa_1 (\{x_1=0\}\cap U_1).$ Shrinking $U_1$ if necessary, we assume that there is $\e>0$ so that 
$$(-\e,L+\e)\times\Sigma \ni (t,q)\mapsto \varphi_t(q)\in V\subset  T^*\!M$$
is a diffeomorphism onto its image, $V$. Then, let 
\begin{gather*}
x_i(\varphi_t(q))=x_i(\kappa_1(q)),\quad 1<i\leq n,\qquad \xi_i(\varphi_t(q))=\xi_i(\kappa_i(q)),\quad 1\leq i\leq n,\\
 x_1(\varphi_t(q))=t.
\end{gather*}
In particular, we have 
$$
H_px_i=0,\quad 1<i\leq n,\qquad H_p\xi_i=0,\quad 1\leq i \leq n,\qquad H_px_1=1.
$$

By Jacobi's identity, for $f,g\in \CR^{\infty}(T^*\!M)$, 
$$
H_p\{g,f\}=\{g,H_pf\}-\{f,H_pg\}.
$$
Therefore, since the corresponding identities hold at $\Sigma$,
$$
\{x_i,x_j\}=0,\qquad \{\xi_i,\xi_j\}=0,\qquad \{\xi_i,x_j\}=\delta_{ij}.
$$
In particular, this implies that $\kappa^{-1}:V\to T^*\RR^n$
$$\kappa^{-1}(q):=(x(q),\xi(q))$$
is a symplectomorphism onto its image, $U$. Furthermore, since $p$ is invariant under $H_p$, $\kappa^*p=\xi_1$ and hence also $\kappa(\Gamma^0_{[0,L]})=\Gamma_{[0,L]}$. 
\end{proof}

\begin{proposition}
\label{p:lower}
Let $\gamma,$ a geodesic, and $U\subset M$ satisfy~\eqref{e:U}. Then for any $\ep>0$ there exists $\chi \in \CR_c^\infty(U)$ and
$f \in L^2$ with $\supp f \subset \{\chi=1\}$ and $u \in L^2_{\loc}$ with
$$
P u=f
$$
and
$$
\WF_h u \subset \bigcup_{s \geq 0} \big[\exp(sH_p) \WF_h f\big].
$$
and where
$$
\norm{\chi u}_{L^2} \geq \bigg( \frac {2 L-\e}{\pi h} 
\bigg) \norm{f}_{L^2}
$$
\end{proposition}

Note that $u$ is purely microlocally outgoing by the wavefront set
statement, hence is given by the outgoing resolvent applied to $f$
(modulo $O(h^\infty)$).
\begin{proof}
We first let $\kappa$ be a symplectomorphism
from $T^*\RR^n$ into $T^*\!M$ with
$$
\kappa: \Gamma^0_{[0,L]} \to \Gamma_{[0,L]},\quad \kappa^* p=\xi_1,
$$
where
$$
\Gamma^0_I=\{x_1 \in I, \xi_1=0, x'=\xi'=0\};
$$
this exists by Lemma~\ref{lemma:longdarboux}.

Now quantize $\kappa$ to a microlocally unitary FIO $T: \CR^\infty_c(\RR^n)
\to \CR^\infty_c(M)$ with parametrix $S$ such that $ST-I, TS-I$ have no
microsupport near $\Gamma^0_{[0,L]}$ resp.\ $\Gamma^0_{[0,L]}$ and such that
$$
T hD_1=PT + O(h^\infty)\ \text{microlocally near } \Gamma^0_{[0,L]}.
$$
(There are no obstructions to such a quantization as long as we work
on a contractible neighborhood of $\Gamma^0$.)
Assume without loss of generality that $T$ is defined on a
$\delta_0$-neighborhood of $\Gamma^0_{[0, L]}$.

Now for any $\delta<\delta_0$
we define on $\RR^n$
$$
f_0 =
\begin{cases}
  \phi(x')  \cos \pi (x_1-\delta)/2(L-2\delta),  & x_1 \in [\delta, L-\delta],\\
  0, & \text{otherwise}
\end{cases}
$$
where $\phi$ is smooth and compactly supported in $B(0, \delta)$, with
$$
\int \abs{\phi(x')}^2 \, dx'=1.
$$
Let
$$
v_0 =\begin{cases}
 0, & x_1<\delta, \\
  h^{-1} \frac{2(L-2\delta)}\pi \phi(x')  \sin \pi (x_1-\delta)/2(L-2\delta),  & x_1 \in [\delta, L-\delta],\\
h^{-1} \frac{2(L-2\delta)}\pi \phi(x'), & x_1>L-\delta;
\end{cases}
$$
Thus of course $h D_1 v_0 = f_0$ as distributions.
Let $\psi(x_1)$ be a smooth cutoff function supported in $(0,
L-\delta/4)$ and equal to $1$ on $[\delta, L-\delta/2].$
Set $u_0=\psi(x_1) v_0,$
$w=T u_0,$
and $f=T f_0$
Then modulo $O(h^\infty),$
\begin{align*}
  Pw&\equiv PT u_0\equiv T hD_{x_1} u_0\equiv T f_0 +Tr_0\equiv f+r  
\end{align*}
where $r_0$ is the error term coming from the $hD_1(\psi)$ term.  Note
that
$$
\WF r_0\subset \big\{ x_1 \in [L-\delta/2, L-\delta/4],\
\smallabs{x'}<\delta, \xi=0\};
$$
which is thus in an arbitrarily small neighborhood of the endpoint of
$\Gamma^0_{[0, L]}.$
Thus by the construction of $T$,
$r\equiv Tr_0$ has wavefront set in an arbitararily small neighborhood
of $\gamma(L)$ (as $\delta\downarrow 0$).

Note also that we may construct a cutoff function $\chi \in \CR_c^\infty(U)$
such that
$$\chi=1 \text{ on } \WF_h f \text{ and }\chi=0 \text{ on }\pi
\{\exp(s H_p)(\WF_h r),\ s\geq 0\}:$$ This follows from our hypotheses
since on $\RR^n$ $f$ is supported on
$x_1 \in [\delta, L-\delta]$ in a small neighborhood of $\Gamma^0$
while $\WF_h r_0 \subset x_1 \in [L-\delta/2, L-\delta/4],$ and the
forward flowout is thus contained (in these local coordinates) in
$\{x_1 \geq L-\delta/2\}.$ Since the Hamilton flow is non-radial and
$x_1$ maps to the flow parameter under $\kappa,$ we know that
$\WF_h r_0$ is separated from a neighborhood of $\Gamma_{[0,L-\delta]}$ in
\emph{base} variables, not just in the cotangent bundle, and shrinking
$\delta$ if necessary we obtain the separation in the base variables.
Note for use below that since $\chi=1$ on $\WF_h f,$ pulling back by
$T$ yields a fortiori $\kappa^* (\chi) =1$ on $\WF_h f_0$ and hence
  \begin{equation}
    \label{eq:1}
    \kappa^*(\chi) =  1 \text{ on } \big\{ x_1 \in [\delta, L-\delta],\
\smallabs{x'}<\delta, \xi=0\}
  \end{equation}

Finally we set
$$
u=w-(P-i0)^{-1} r
$$
so that
$$
Pu=f.
$$
Owing to the propagation of singularities for the outgoing
resolvent, $\WF_h (u-w)$ lies in the \emph{forward flowout} of $\WF_h
r;$ by construction this $\WF_h r$ is  disjoint from the support of $\chi$ and by
our geometric hypotheses, its forward flowout \emph{remains} disjoint from
$U:$ here we use the fact that $\pi \gamma(L+s)\notin U$ for $s>0$ and
indeed this point escapes to infinity, hence $\pi \exp (sH_p)
(\gamma(L)) \notin \WF_h f$ for all $s>0.$  By continuity, the same is
true with $\gamma(L)$ replaced by any point in $\WF_h r$ for $\delta$
sufficiently small.  Thus,
$$
\chi u = \chi w +O(h^\infty).
$$
By microlocal unitarity of $T$ we compute by the Egorov Theorem
\begin{align*}
  \norm{\chi u}^2 &= \norm{\chi w}^2+O(h^\infty)\\
  &= \norm{S\chi TS w}^2+O(h^\infty)\\
                  &= \norm{\Op(\kappa^* \chi) u_0}^2+O(h)\\
                  &\geq \int_\delta^{1-\delta} \int_{B(0, \delta)}
  \smallabs{v_0}^2 dx' \,
                  dx_1  +O(h)\\
                  &\geq \int_\delta^{1-\delta} \int_{B(0, \delta)}
  h^{-2} \frac{4(L-2\delta)^2}{\pi^2} \abs{\phi(x') }^2 \sin^2 \pi x_1/2(L-\delta)
\, dx' \,
                    dx_1  +O(h)\\
  &= \frac\pi 4 \frac{1}{h^2\mu^3} + O(h),
\end{align*}
where
$$
\mu=\frac{\pi}{2(L-2\delta)};
$$
here we have used the fact that $\kappa^*\chi$ equals $1$ on the zero
section over $[\delta, 1-\delta] \times B(0,\delta),$ while the
semiclassical wavefront set of $u_0$ lies within the zero section,
hence the last inequality follows by existence of a microlocal
parametrix for $\Op (\kappa^*\chi)$ on this set.

Meanwhile,
\begin{align*}
\norm{f}^2 &=\int_\delta^{1-\delta} \int_{B(0, \delta)}
                  \abs{\phi(x')}^2  \cos^2 \pi x_1/2(L-2\delta)\, dx' \,
                  dx_1  = \frac \pi {4} \frac 1 \mu.
\end{align*}
Thus
\begin{align*}
  \frac{\norm{\chi u}}{\norm{f}} &\geq \frac{1}{h \mu}= \frac{2(L-2\delta)}{\pi h}.
\end{align*}
\end{proof}

 Finally, we apply Proposition~\ref{p:lower} in the case of a
  nontrapping manifold with Euclidean ends and $P=-h^2\Delta_g-1$.
  Let $R_1<R'<R''$ so that $B(0,R'')\subset \{\chi \equiv 1\}$. By the
  definition of $L(g,R'')$, there is a null-bicharacteristic
  (a speed-$2$ geodesic lifted to $S^*M$) $\gamma$ with
  $\pi \gamma(0)\in \partial B(0,R'')$, $\pi \gamma(s)\in B(0,R'')$ for
  $s\in (0,L(g,R''))$ and $\gamma(s)\notin B(0,R'')$ for
  $s>L(g,R'')$. In particular, since $L(g,R')<L(g,R'')$,
  Proposition~\ref{p:lower} applies to $\gamma$ with $U=B(0,R'')$ and
  $\e<2L(g,R'')-2L(g,R')$. In particular, there exist $f \in L^2(M)$
  with $\supp f\subset B(0,R'')$ so that
\begin{equation}
\label{e:wf}
Pu=hf,\qquad \qquad \WF_h u \subset \bigcup_{s\geq0} [\exp(sH_p)\WF_h f]
\end{equation}
and
$$
\|u\|_{L^2(B(0,R''))}\geq \frac{2L(g,R'')-2L(g,R'')+2L(g,R')}{\pi}\|f\|_{L^2}=\frac{2L(g,R')}{\pi }\|f\|_{L^2}.
$$ 
Next, observe that ~\eqref{e:wf} implies that 
$$
u=(-h^2\Delta_g-1-i0)^{-1}hf+O(h^\infty\|f\|_{L^2})_{L_{\loc}^2}
$$
In particular, letting $k=h^{-1}$, 
$$
\|\chi (-\Delta_g-k^2-i0)^{-1}f\|_{L^2}\geq \left(\frac{2L(g,R')}{\pi k }-C_Nk^{-N}\right)\|f\|_{L^2}
$$
completing the proof of the lower bound in Theorem~\ref{theorem:main}.

\section{Application to numerical analysis of the finite-element method}\label{sec:FEM}

\subsection{Summary}

In this section, we focus on the implications the bound \eqref{Sobolevupperbound2} has on the numerical analysis of solving the Helmholtz exterior Dirichlet problem by the finite-element method (FEM).  We mention two other numerical-analysis applications of  \eqref{Sobolevupperbound2} in Remarks \ref{rem:precondition} and \ref{rem:UQ} at the end.

We consider the \emph{h-version} of the FEM; i.e.~the solution is
approximated in spaces of piecewise polynomials of fixed degree on a
mesh with meshwidth $\hFEM$ (we use this notation to avoid a notational clash with the semiclassical parameter $h:=k^{-1}$ in the rest of the paper).
  The question of how fast $\hFEM$ must decrease with $k$ to maintain accuracy as $k\tendi$ was thoroughly investigated in the case of the \emph{constant-coefficient} Helmholtz equation (i.e.~\eqref{eq:PDE} below with $A=I$ and $\coeffn=1$) by \cite{IhBa:95a, IhBa:97} when $d=1$ and by \cite{Me:95, MeSa:10, MeSa:11, EsMe:12} when $d=2, 3$. For example, in the case of piecewise-linear polynomials, and when the 
solution of the boundary value problem is nontrapping, 
  the FEM satisfies a \emph{quasioptimal} error estimate (of the form \eqref{eq:QO} and \eqref{eq:QOS} below) when 
\beq\label{eq:E1}
\hFEM k^2 \leq c
\eeq
for a sufficiently small constant $c$ (the case when the boundary value problem is trapping is more complicated; see \cite[\S1.4]{ChSpGiSm:17} for some initial results).

In this section, we use the bound \eqref{Sobolevupperbound2}  to prove the analogue of the result above in the case of the \emph{variable-coefficient} Helmholtz equation (\eqref{eq:PDE} below) posed in the exterior of a nontrapping Dirichlet obstacle (see Theorem \ref{thm:Emain} below). In particular, we show how the constant $c$ in \eqref{eq:E1} depends on the constant in \eqref{Sobolevupperbound2}. The key point is that our bound on $\hFEM$ shows explicitly how the constant $c$ in \eqref{eq:E1} \emph{decreases} (i.e.~the requirement on $\hFEM$ for quasioptimality becomes more stringent) as the length of the longest ray \emph{increases}.

\subsection{The Helmholtz exterior Dirichlet problem and FEM set-up}\label{sec:EDPsetup}

\hspace{1ex}

\

\paragraph{{\bf Motivation from applications}}
Several important physical applications involve the PDE
\beq\label{eq:PDE}
\nabla\cdot(\coeffA \gu ) + k^2 \coeffn u = -f
\eeq
posed in $\Rea^n\setminus{\Omega}$, 
where $A$ is a symmetric positive-definite matrix-valued function of position and $\coeffn$ is a strictly-positive scalar-valued function of position.
One example is via reductions of the time-harmonic Maxwell equations
\beq\label{eq:Maxwell}
{\rm curl} H + \ri k \eps E = (\ri k)^{-1} J, \quad {\rm curl} E - \ri k \mu H=0 \quad \text{ in } \Rea^3,
\eeq
when all the fields involved depend only on two Cartesian space variables, say $x$ and $y$.
In the transverse-magnetic (TM) mode, $J$ and $E$ are given by $J = (0,0,J_z(x,y))$ and $E = (0,0, E_z(x,y))$ so, when additionally the permittivity $\eps$ is a scalar and the permeability $\mu$ satisfies 
\beq\label{eq:mu}
\mu = \left(
\begin{array}{cc}
\widetilde{\mu} & 0\\
0&1 
\end{array}
\right)
\eeq
for $\widetilde{\mu}$ a $2\times2$ symmetric positive-definite matrix, 
$E_z$ satisfies \eqref{eq:PDE} with $\coeffn=\eps$,
\beq\label{eq:TEA}
A = 
\left(
\begin{array}{cc}
0 & 1\\
-1&0
\end{array}
\right)^T
\big(\widetilde{\mu}\big)^{-1}
\left(
\begin{array}{cc}
0 & 1\\
-1&0
\end{array}
\right),
\eeq
and $f= J_z$.
Similarly, 
in the case of the transverse-electric (TE) mode, $J$ and $H$ are given by $J = (J_x(x,y), J_y(x,y),0)$ and $H = (0,0, H_z(x,y))$, so that when $\mu$ is a scalar and $\eps$ satisfies an equation analogous to \eqref{eq:mu}, $H_z$ satisfies \eqref{eq:PDE} with $A$ given by \eqref{eq:TEA} with $\widetilde{\mu}$ replaced by $\widetilde{\eps}$, $\coeffn= \mu$, and
\beqs
f= -\frac{1}{\ri k} \nabla\cdot \left[
\left(
\begin{array}{cc}
0 & 1\\
-1&0
\end{array}
\right)
\big(\widetilde{\eps}\big)^{-1}
\left(
\begin{array}{c}
J_x\\
J_y
\end{array}
\right)
\right].
\eeqs 
Additionally, the so-called acoustic-approximation of the elastodynamic wave equation is \eqref{eq:PDE} with $A$ the inverse of the (scalar) density and $\coeffn$ the inverse of the (scalar) bulk modulus; see, e.g., the derivation in \cite[\S1.2.6]{Ch:15}.

\

\paragraph{{\bf Placing the PDE \eqref{eq:PDE} in the framework of \S\ref{sec:moregeneral}}}
We let $M:= \Rea^n\setminus \Omega$, where $\Omega$ is a compact set such that its complement is connected,
\beqs
g^{ij} := \frac{A_{ij}}{\coeffn}, \quad L_j:= \sum_{i=1}^n \frac{g^{ij}}{\sqrt{{\rm det} \,g}} \partial_i \big( \sqrt{{\rm det} \,g}\big)
+\sum_{i=1}^n\partial_i(g^{ij}) - \sum_{i=1}^n\frac{1}{\coeffn}\partial_i (A_{ij})
, \quad L:=0,
\eeqs
and
\beqs
P(g)\phi:= -\frac{1}{\coeffn}\sum_{i,j=1}^n\partial_i \big(A_{ij}\partial_j \phi),
\eeqs
and then observe that \eqref{eq:convert1} is satisfied, since
\beqs
\Delta_g \phi := \sum_{i,j=1}^n \frac{1}{\sqrt{ {\rm det}\, g}} \partial_i \big(\sqrt{ {\rm det}\, g}\, g^{ij} \, \partial_j \phi\big).
\eeqs
With these definitions, the PDE \eqref{eq:PDE} is $(P(g) - k^2)u=f/\nu$.
The operator $P_\Omega$ is then self-adjoint with respect to $L^2(\Rea^n\setminus\Omega;\nu)$ and the $H^1$ norm defined by \eqref{eq:weightednorm} becomes
\beq\label{eq:H1norm}
\N{\phi}^2_{H^1_\coeffn (\Rea^n\setminus \Omega)}:= \int_{\Rea^n \setminus \Omega} \big( (A \nabla \phi)\cdot\overline{\nabla \phi} + k^2 \coeffn|\phi|^2\big)\rd \bx.
\eeq
\emph{In this section (\S\ref{sec:FEM}) only}, to make contact with the standard numerical-analysis literature, \emph{we use non-semiclassically scaled Sobolev spaces}; i.e., the norms on the spaces $H^s$ do \emph{not} include powers of $k$ unless otherwise indicated. Indeed, in this section we write \eqref{eq:H1norm} as the weighted $H^1$ norm
\beq\label{eq:1knormg}
\N{\phi}^2_{H^1_{k,A,\coeffn}(\Rea^n\setminus\Omega)} := \big\|\coeffA^{1/2}\nabla\phi\big\|^2_{L^2(\Rea^n\setminus\Omega)} + k^2 \big\|{\coeffn}^{1/2}\phi\big\|^2_{L^2(\Rea^n\setminus\Omega)} ,
\eeq
and we use analogous notation for norms over subsets of $\Rea^n\setminus\Omega$. 
We also define the norms
$$
\|\phi\|_{H^m(\re^n\setminus \Omega)}^2:=\sum_{|\alpha|\leq m}\int_{\re^n\setminus \Omega} |\partial^\alpha \phi|^2\,\rd x;
$$
i.e., if there are no $A,\nu$ in the subscript, the $H^m$ norm is the `standard' $H^m$ norm on $\re^n$.

By \eqref{eq:L}, $L(\coeffn\coeffA^{-1}, \Omega, R)$ is the length of the longest generalized bicharacteristic (informally, ray) in $\ballR\setminus \Omega$. Recall that, in the case $\Omega=\emptyset$, the (generalized) bicharacteristics are the projections in $\bx$ of the solutions $(\bx(s), \bxi(s)) \in \Rea^n\times \Rea^n$ of the Hamiltonian system
\beq\label{eq:rays}
\diff{x_i}{s}(s) = \pdiff{}{\xi_i}H\big(\bx(s), \bxi(s) \big), \qquad
\diff{\xi_i}{s}(s)
 = -\pdiff{}{x_i}H\big(\bx(s), \bxi(s) \big),
\eeq
where the Hamiltonian $H(\bx,\bxi)$ is given by 
\beq\label{eq:Hamiltonian}
H(\bx,\bxi):= \frac{1}{\coeffn(\bx)}\sum_{i=1}^n\sum_{j=1}^{n} A_{ij}(\bx)\xi_i \xi_j - 1.
\eeq

\paragraph{{\bf Exterior Dirichlet problem and its variational formulation}}

With $\Omega$  a compact set such that its complement is connected, we define $\domain_+:= \Rea^n\setminus \Omega$. 
Since the vast majority of numerical-analysis applications of the Helmholtz equation are in two and three dimensions, we restrict attention to $d=2,3$.
Let $\gamma :H^1_{\rm loc}(\Omega) \rightarrow H^{1/2}(\partial\Omega)$ be the trace operator.
Let $A$ be a symmetric positive-definite matrix-valued function of position such that $\supp(I-A)$ is a compact subset of $\Rea^n$. 
Let $\coeffn$ be a strictly-positive scalar-valued function of position such that $\supp(1-\coeffn)$ is a compact subset of $\Rea^n$.
Let $A_{\min}$ and $A_{\max}$ be such that 
\beq\label{eq:Amin}
A_{\min} \leq A(\bx)\leq A_{\max}\quad\tfa \bx\in \domain_+, \text{ in the sense of quadratic forms.}
\eeq
and $\coeffn_{\min}$ and $\coeffn_{\max}$ be such that 
\beq\label{eq:nmin}
\coeffn_{\min} \leq \coeffn(\bx)\leq \coeffn_{\max}\quad\tfa \bx\in \domain_+.
\eeq
We consider solving \emph{both} the exterior Dirichlet problem 
\begin{subequations}\label{eq:EDP1}
\beq
\nabla\cdot(\coeffA \gu ) + k^2 \coeffn u = -f \quad \tin \domain_+,
\eeq
\beq
\trace u =0 \quad\ton \partial \Omega, \quad\tand
\eeq
\beq\label{eq:src}
\pdiff{u}{r}(\bx) - \ri k u(\bx) = o \left( \frac{1}{r^{(d-1)/2}}\right) \quad \text{as $r:= |\bx|\tendi$, uniformly in $\widehat{\bx}:= \bx/r$},
\eeq
\end{subequations}
\emph{and} the sound-soft scattering problem 
\begin{subequations}\label{eq:EDP2}
\beq
\nabla\cdot\big(\coeffA \gu \big) + k^2 \coeffn u =0 \quad \tin \domain_+,
\eeq
\beq
\trace u =0 \quad\ton \partial \Omega, \tand
\eeq
\beq
 (u-u^I) \text{ satisfies the radiation condition \eqref{eq:src}},
 \eeq
 \end{subequations}
where $u^I$ is solution of $\Delta u^I +k^2 u^I=0$ (such as a plane wave or point source) that is smooth in a neighborhood of $\supp(I-\coeffA)\cup\supp(1-\coeffn)\cup \Omega$.

The standard variational formulations of these problems are posed on $\domain_R:= \domain_+\cap \ballR = \ballR \setminus \Omega$, where $R$ is chosen large enough such that $\supp (I-\coeffA)$, $\supp(1-\coeffn)$, and $\supp\, f$ are all compactly contained in $\ballR$, and also such that $u^I$ is smooth in $\ballR$.
Let 
\beq\label{eq:spaceEDP}
H_{0,D}^1(\domain_R):= \big\{ v\in H^1(\domain_R) : \trace v=0 \ton \partial \Omega\big\},
\eeq
let $\Gamma_R:= \partial \ballR$, and 
let $T_R: H^{1/2}(\Gamma_R) \rightarrow H^{-1/2}(\Gamma_R)$ be the Dirichlet-to-Neumann (DtN) map for the equation $\Delta u+k^2 u=0$ posed in the exterior of $\ballR$ with the Sommerfeld radiation condition \eqref{eq:src}; the definition of $T_R$ in terms of Hankel functions and polar coordinates (when $d=2$)/spherical polar coordinates (when $d=3$) is given in, e.g., \cite[Equations 3.5 and 3.6]{ChMo:08} \cite[\S2.6.3]{Ne:01}, \cite[Equations 3.7 and 3.10]{MeSa:10}. 
The variational formulation of \eqref{eq:EDP1} is:
\beq\label{eq:EDPvar}
\text{ find } u \in H^1_{0,D}(\domain_R) \tst\,\, a(u,v)=F(v) \,\,\tfa v\in H^1_{0,D}(\domain_R),
\eeq
where
\beq\label{eq:EDPa}
a(u,v):= \int_{\domain_R} 
\Big((\coeffA\gu)\cdot\gvb
 - k^2 \coeffn u\vb\Big)\,\rd \bx - \langle T_R (\trace u),\trace v\rangle_{\Gamma_R} \quad\tand\quad F(v):= \int_{\domain_R} f\, \vb\, \rd \bx,
\eeq
and where $\langle\cdot,\cdot\rangle_{\GR}$ denotes the duality pairing on $\GR$ that is linear in the first argument and antilinear in the second.
The variational formulation of \eqref{eq:EDP2} is \eqref{eq:EDPvar} but with $F(v)$ instead given by 
\beqs
F(v):= 
\int_{\Gamma_R} \left( \pdiff{u^I}{r} - T_R u^I\right)\trace\vb \, \rd s.
\eeqs

\

\paragraph{{\bf Finite-element method}}

Let $\cT_\hFEM$ be a family of triangulations of $\Omega_R$ (in the sense of, e.g., \cite[Page 67]{Ci:91})  that is shape regular (see, e.g., \cite[Definition 4.4.13]{BrSc:08}, \cite[Page 128]{Ci:91}).
 Let $$\cH_\hFEM:= \{v \in \CR(\overline{\Omega_R}) : v|_K \text{ is a polynomial of degree 1 for each } K \in \cT_{\hFEM}\},$$ and observe that the dimension of $\cH_\hFEM$ is proportional to $\hFEM^{-n}$.
 Our main result, Theorem \ref{thm:Emain} below, is valid when $\Omega$ is $\CR^{\infty}$ (or, more precisely, $\CR^m$ for some large $m$ not given explicitly). For such $\Omega$ it is not possible to fit $\partial \Omega$ exactly with simplicial elements
(i.e.~when each element of $\cT_{\hFEM}$ is a simplex), and fitting $\partial \Omega$ with isoparametric elements (see, e.g, \cite[Chapter VI]{Ci:91}) or curved elements (see, e.g., \cite{Be:89}) is impractical, and therefore  
some analysis of non-conforming error is necessary; since this is very standard (see, e.g., \cite[Chapter 10]{BrSc:08}), we ignore this issue here.

The finite-element method for the variational problem \eqref{eq:EDPvar} is the Galerkin method applied to the variational problem \eqref{eq:EDPvar}, i.e.
\beq\label{eq:EDPFEM}
\text{ find } u_\hFEM \in \cH_\hFEM\tst\,\, a(u_\hFEM,v_\hFEM)=F(v_\hFEM) \,\, \tfa v_\hFEM\in  \cH_\hFEM.
\eeq

\subsection{Main result}

Before stating the main result (Theorem \ref{thm:Emain} below) we define the following constants -- all independent of $k$ but dependent on one or more of $A$, $\coeffn$, $\Omega$, $R$, and $k_0$ -- upon which the main result depends.

\

\paragraph{$\Cint$:}Recall that the \emph{nodal interpolant} $I_\hFEM : \CR(\overline{\Omega_R}) \rightarrow \cH_\hFEM$ is well-defined for functions in $H^2(\Omega_R)$ (for $d=2,3$) and satisfies
\beq\label{eq:E2}
\big\|\coeffn^{1/2}(v- I_\hFEM v)\big\|_{L^2(\Omega_R)} + \hFEM \big\|\coeffA^{1/2}\nabla(v- I_\hFEM v)\big\|_{L^2(\Omega_R)} \leq \Cint (\hFEM)^2 \N{v}_{H^2(\Omega_R)} ,
\eeq
for all $v\in H^2(\Omega_R)$, for some constant $\Cint = \Cint(\coeffA,\coeffn)$. 
The only reason $\Cint$ depends on $A$ and $\coeffn$ is because of the weights $A$ and $\coeffn$ in the norms on the left-hand side of \eqref{eq:E2}. Indeed, by, e.g., \cite[Equation 4.4.28]{BrSc:08}, 
\beqs
\N{v- I_\hFEM v}_{L^2(\Omega_R)} + \hFEM \N{\nabla(v- I_\hFEM v)}_{L^2(\Omega_R)} \leq \Ctint (\hFEM)^2 \N{v}_{H^2(\Omega_R)} ,
\eeqs
for all $v\in H^2(\Omega_R)$, for some $\Ctint$ that depends only on the shape-regularity constant of the mesh, and thus \eqref{eq:E2} holds with 
\beqs
\Cint(\coeffA,\coeffn):=\Ctint\max\left\{ (A_{\max})^{1/2}, (\coeffn_{\max})^{1/2}\right\}.
\eeqs

\paragraph{$\CTR$:} There exists $\CTR= \CTR(A, \coeffn, R,k_0)$ such that 
\beq\label{eq:TR2}
\big|\big\langle T_R(\trace u), \trace v\rangle_{\Gamma_R}\big\rangle\big| \leq \CTR \N{u}_{\HoDkkg}  \N{v}_{\HoDkkg}  \quad \tfa u, v \in H^1(\domain_R) \text{ and for all } k\geq k_0;
\eeq
see \cite[Lemma 3.3]{MeSa:10}.
As above, the only reason $\CTR$ depends on $A$ and $\coeffn$ is because of the weights $A$ and $\coeffn$ in \eqref{eq:1knormg}.
Indeed, \cite[Lemma 3.3]{MeSa:10} bounds the left-hand side of \eqref{eq:TR2} by $H^{1/2}(\GR)$ and $L^2(\GR)$ norms of $\trace u$ and $\trace v$, and then uses trace theorems to prove that
\beqs
\big|\big\langle T_R(\trace u), \trace v\rangle_{\Gamma_R}\big\rangle\big| \leq 
\CtTR(R,k_0)\|u\|_{H^1_{k, I,1}(\Omega_R)}\|v\|_{H^1_{k, I,1}(\Omega_R)},
\eeqs
for some  
$\CtTR(R,k_0)$ given explicitly in the proof of \cite[Lemma 3.3]{MeSa:10}. Therefore, \eqref{eq:TR2} holds with 
\beqs
\CTR(A, \coeffn, R,k_0) := \CtTR(R,k_0) \max\left\{ \frac{1}{(A_{\min})^{1/2}}, \frac{1}{(\coeffn_{\min})^{1/2}}\right\}.
\eeqs
\

\paragraph{$\Creg$:} There exists $\Creg=\Creg(A,  \Omega, R)$ such that, if $\widetilde{f}\in L^2(\Omega_{R+1})$ and $v\in H^1(\Omega_{R+1})$ satisfy
\beqs
\nabla \cdot (\coeffA\nabla v) = \widetilde{f} \quad \tin \Omega_{R+1}\tand \trace v = 0 \ton \partial \Omega,
\eeqs
then 
\beq\label{eq:H2}
\N{v}_{H^2(\Omega_R)} \leq \Creg\Big( \|\coeffA^{1/2}\nabla v\|_{L^2(\Omega_{R+1})} + \N{v}_{L^2(\Omega_{R+1})} + \|\widetilde{f}\|_{L^2(\Omega_{R+1})}\Big);
\eeq
since $A\in \CR^{0,1}$ and $\Omega \in \CR^{1,1}$, such a $\Creg$ exists by, e.g., \cite[Theorem 4.18]{Mc:00} and $\Creg$ depends on $A_{\min}$ in \eqref{eq:Amin}, the $\CR^{0,1}$ norm of $A$, and the $\CR^{1,1}$ norm of the parametrisation of $\partial\Omega$. Note that the $R+1$ in the norms on the right-hand side of \eqref{eq:H2} can be replaced by $R+b$ for any $b>0$, but the constant $\Creg$ blows up as $b\tendo$.

\begin{theorem}[Quasioptimality of the Galerkin method]\label{thm:Emain}
Let $\Omega$, $A$, and $\coeffn$ be as in \S\ref{sec:EDPsetup}; i.e.~$\Omega$ a compact set such that its complement $\domain_+$ is connected,
$A$ a symmetric positive-definite matrix-valued function of position such that $\supp(I-A)$ is a compact subset of $\ballR$
and $\coeffn$ a strictly-positive scalar-valued function of position such that $\supp(1-\coeffn)$ is a compact subset of $\ballR$.
Furthermore, let $\Omega$ be $\CR^{\infty}$, and let $A, \coeffn \in \CR^{1,1}(\Omega_R)$ be such that $A$ and $\coeffn$  are $\CR^{\infty}$ in a neighborhood of $\partial \Omega$. Finally assume that
$\Omega$ is nowhere tangent to infinite order to the geodesic flow generated by $A$ and $\coeffn$ via the Hamiltonian \eqref{eq:Hamiltonian}, and that this flow on $\Omega_+$ is nontrapping.

There exists a $k_0>0$ such that if
\begin{align}\label{eq:mesh}
1&\geq 
\hFEM k^2 
\sqrt{1 + (\hFEM k)^2}\, L\big(\coeffn\coeffA^{-1},\Omega,R+2\big)\,
\Cint\Creg (1+ \CTR)\,
\left(\frac{\coeffn_{\max}}{\coeffn_{\min}}\right)^{1/2} \frac{4\sqrt{2}}{\pi} 
\\ &\hspace{5cm}\times
\left( (\coeffn_{\max})^{1/2}+ \frac{1+(\coeffn_{\min})^{1/2}}{k_0} + \frac{1}{k_0^2(\coeffn_{\min})^{1/2}}
\right),\nonumber
\end{align}
then the Galerkin equations \eqref{eq:EDPFEM} have a unique solution which satisfies 
\beq\label{eq:QO}
\N{u-u_\hFEM}_{\HoDkkg} \leq 2\big(1+ \CTR\big)\left(\min_{v_\hFEM\in\cH_\hFEM} \N{u-v_\hFEM}_{\HoDkkg}\right).
\eeq
(Note that $k_0$ depends on $A$, $\coeffn$, and $\Omega$ but---as in Theorem \ref{thm:2}---is uniform on small $\CR^{2,\alpha}$ neighborhoods of $A$ and $\coeffn$; see Section \ref{sec:uniform}.)
\end{theorem}

\bre[The mesh threshold \eqref{eq:mesh}] If one assumes that
$\hFEM k \leq C$, then the mesh threshold \eqref{eq:mesh} can be
written in the form \eqref{eq:E1}, i.e.~$\hFEM k^2 \leq c$.  The key
point is that if $\coeffA$, $\coeffn$, and $\Omega$ are as in the
statement of the theorem with the $\CR^{0,1}$ norms of $A$ and
  $\coeffn$ and the $\CR^{1,1}$ norm of $\Omega$ all bounded by
$\widetilde{C}$, say, then all the constants in $c$ apart from
$L(\coeffn\coeffA^{-1}, \Omega, R+2)$ (namely, $A_{\max}$,
$\Cint, \Creg, \CTR$) are bounded in terms of $\widetilde{C}$,
$A_{\min}$ and $\coeffn_{\min}$, but
$L(\coeffn\coeffA^{-1}, \Omega, R+2)$ can be arbitrarily
large. Furthermore, $c$ \emph{decreases} as
$L(\coeffn\coeffA^{-1},\Omega,R+2)$ \emph{increases}; i.e., the
condition on the mesh threshold for quasioptimality becomes more
restrictive as the length of the longest ray grows.  \ere 

\bre[Quasioptimality for higher-order finite-element spaces]
We have only considered the lowest-order conforming finite-element space of $H^1$, namely continuous piecewise-linear polynomials. In the case of the constant-coefficient Helmholtz equation, the mesh threshold for quasioptimality (analogous to \eqref{eq:E1}) has been determined by \cite{MeSa:10, MeSa:11, EsMe:12} for spaces of arbitrary order $p$ (possibly depending on $k$), with the threshold then explicit in $k$, $h$, and $p$. These results are obtained by a careful splitting of the solution that does not immediately generalise to the case $A\neq I$. However, in the case of $p$ fixed, the mesh thresholds for quasioptimality from \cite{MeSa:10, MeSa:11, EsMe:12} have recently been obtained in \cite{ChNi:18} using a simpler method (relying on well-known elliptic regularity results such as \eqref{eq:H2}), although the constants are not given explicitly in $p$. In the case when the DtN map is approximated by an impedance boundary condition, the results of \cite{ChNi:18} immediately apply to the case $A\neq I$, and they can in principle be extended to the full scattering problem considered here.
\ere

\bre[Comparison to existing results in the literature] The only
existing results in the literature on quasioptimality (explicit in all
parameters and coefficients) for the Galerkin method applied to the
variable-coefficient Helmholtz equation are in \cite{Ch:16} and
\cite{GrSa:18}. The main similarity between these two works and the
present paper is that they all use the ``Schatz argument" described in
Lemma \ref{lem:Schatz} below (and pioneered for Helmholtz problems in
\cite{Sa:06, BaSa:07}). The main differences between \cite{Ch:16,
  GrSa:18} and the present paper are that (i) both \cite{Ch:16} and
\cite{GrSa:18} consider the Helmholtz equation in a bounded domain
(\cite{Ch:16} in 1-d, \cite{GrSa:18} in 1-, 2-, or 3-d) with an impedance
boundary condition (recall that this is
the simplest-possible approximation to the DtN map operator $T_R$),
and (ii) to get an a priori bound explicit in $k$ and the
coefficients, both \cite{Ch:16} and \cite{GrSa:18} impose conditions
on the coefficients and the domain that are \emph{stronger} than the
analogue of nontrapping for the interior impedance problem (i.e., the
assumption that every ray reaches the boundary in a uniform
time). 
  \ere

\subsection{Proof of Theorem \ref{thm:Emain}}

The heart of the proof is Lemma \ref{lem:Schatz} below. This gives a condition for quasi-optimality to hold in terms of how well the solution of the adjoint problem 
is approximated by the finite-element space, and relies on the fact that $a(\cdot,\cdot)$ satisfies a G\aa rding inequality. This argument essentially goes back to Schatz \cite{Sc:74} (using the Aubin-Nitsche technique; see, e.g., the references in \cite[Remark 26]{Sp:15}) and was extensively used in the analysis of the FEM for Helmholtz problems by \cite{Sa:06, MeSa:10, MeSa:11, EsMe:12}.

Before stating Lemma \ref{lem:Schatz} we need to introduce some notation.
Let $\Ccont = \Ccont(A, \coeffn, R, k_0)$ be the \emph{continuity constant} of the sesquilinear form $a(\cdot,\cdot)$ (defined in \eqref{eq:EDPa}) in the norm $\|\cdot\|_{\HoDkkg}$; i.e.
\beqs
a(u,v)\leq \Ccont \N{u}_{\HoDkkg} \N{v}_{\HoDkkg} \quad\tfa u, v \in H^1_{0,D}(\Omega_R).
\eeqs
By the Cauchy-Schwarz inequality and \eqref{eq:TR2} we have 
\beq\label{eq:Cc}
\Ccont \leq 1 + \CTR.
\eeq

\begin{definition}[The adjoint sesquilinear form $a^*(\cdot,\cdot)$]\label{def:adjoint}
The adjoint sesquilinear form, $a^*(u,v)$, to the sesquilinear form $a(\cdot,\cdot)$ defined in \eqref{eq:EDPa} is given by
\beq\label{eq:EDPadjoint}
a^*(u,v) := \overline{a(v,u)}= \int_{\domain_R} 
\Big((\coeffA \gu)\cdot\gvb
 - k^2 \coeffn u\vb\Big) - \big\langle \gamma u,T_R(\gamma v)\big\rangle_{\Gamma_R}.
\eeq
\end{definition}

\ble
\label{lem:adjoint}
Given $F\in (\HoDk)'$, if $u$ is the solution to the variational problem
\beq\label{eq:adjoint1}
a^*(u,v)= F(v) \quad\tfa v\in H^1_{0,D}(\OR),
\eeq
then $\overline{u}$ satisfies
\beq\label{eq:adjoint2}
a(\overline{u},v)= \overline{F(\overline{v})} \quad\tfa v\in H^1_{0,D}(\OR).
\eeq
\ele

\bpf[Proof of Lemma \ref{lem:adjoint}]
By \eqref{eq:adjoint1},
\beqs
\overline{a^*(u,\overline{v})}= \overline{F(\overline{v})} \quad\tfa v\in H^1_{0,D}(\OR).
\eeqs
The result then follows from the definition of $a^*(\cdot,\cdot)$ and the following property of the DtN map $T_R$:
\beqs
\big\langle T_R\psi, \overline{\phi} \big\rangle_\Gamma = \big\langle T_R \phi, \overline{\psi}\big\rangle_\Gamma \quad\tfa \phi,\psi \in H^{1/2}(\GR).
\eeqs
This property follows from the fact that, if $u$ and $v$ are solutions of the homogeneous Helmholtz equation $\Delta u +k^2 u=0$ in $\Rea^n\setminus \overline{\ballR}$, both satisfying the Sommerfeld radiation condition \eqref{eq:src}, then
\beqs
\int_{\GR} (\gamma u)\, \pdiff{v}{n} = \int_{\GR} (\gamma v)\,\pdiff{u}{n}
\eeqs
by Green's second identity; see, e.g., \cite[Lemma 6.13]{Sp:15}.
\epf

\ble[Conditions for quasi-optimality]\label{lem:Schatz}
Given $f\in L^2(\Omega_R)$, let $S^*f$ be the solution of the adjoint problem \eqref{eq:adjoint1} with $F(v)$ defined in \eqref{eq:EDPa}. Let
\beq\label{eq:etadef}
\eta(\cH_\hFEM):= \sup_{0\neq f\in L^2(\Omega_R)}\min_{v_\hFEM\in\cH_\hFEM} \frac{\N{S^*f-v_\hFEM}_{\HoDkkg}}{\big\|
f\big\|_{L^2(\Omega_R)}}.
\eeq
If 
\beq\label{eq:S1}
\eta(\cH_\hFEM) \leq \frac{1}{2 \Ccont {(\coeffn_{\max})}^{1/2} k}
\eeq
then the Galerkin equations \eqref{eq:EDPFEM} have a unique solution which satisfies
\beq\label{eq:QOS}
\N{u-u_\hFEM}_{\HoDkkg} \leq 2\Ccont\left(\min_{v_\hFEM\in\cH_\hFEM} \N{u-v_\hFEM}_{\HoDkkg}\right).
\eeq
\ele

\bpf
Since
\beq\label{eq:TR}
\Re \big( - \langle T_R \phi,\phi\rangle_{\Gamma_R}\big) \geq 0 \quad\tfa \phi \in H^{1/2}(\Gamma_R) \text{ and for all } k\geq k_0
\eeq
(see \cite[Corollary 3.1]{ChMo:08} or \cite[Theorem 2.6.4]{Ne:01}), $a(\cdot,\cdot)$ satisfies the G\aa rding inequality
\beqs
\Re\big(a(v,v)\big) \geq \N{v}^2_{\HoDkkg} - 2 k^2 \coeffn_{\max}\N{v}^2_{L^2(\Omega_R)}
\eeqs
and the result follows from the account of the Schatz argument in, e.g., \cite[Theorem 6.32]{Sp:15}. 
\epf

The condition \eqref{eq:S1} on $\eta(\cH_\hFEM)$ is implicitly a condition on the meshwidth $\hFEM$. 
To make this condition explicit, we observe that 
the polynomial-approximation bound \eqref{eq:E2} implies that a bound on $\eta(\cH_\hFEM)$ (defined by \eqref{eq:etadef}) can be obtained from  a bound on the $H^2$ norm of the (adjoint of the) exterior Dirichlet problem.

\ble[Bound on $H^2$ norm]\label{lem:H2}
If $\Omega$, $A$, and $\coeffn$ are as in Theorem \ref{thm:Emain}, then 
there exists a $k_0>0$ such that the solution of the exterior Dirichlet problem with $f$, $I-\coeffA$, and $1-\coeffn$ supported in $\Omega_R$ satisfies
\beq\label{eq:H2final}
\N{u}_{H^2(\Omega_R)} \leq k\, \Creg \frac{2\sqrt{2}}{\pi (\coeffn_{\min})^{1/2}}L\big(\coeffn\coeffA^{-1},\Omega,R+2\big) \left( (\coeffn_{\max})^{1/2} + \frac{1+ (\coeffn_{\min})^{1/2}}{k_0} + \frac{1}{k_0^2(\coeffn_{\min})^{1/2}}
\right) \N{f}_{L^2(\Omega_R)}.
\eeq
\ele
\bpf
Choosing $\chi$ such that $\supp \chi \subset B_{R+2}$ but $\chi =1 $ on $B_{R+1}$, Theorem \ref{thm:2} with $s=1$ implies that there exists $k_0>0$ such that
\beq\label{eq:corollary1}
\big\|\coeffA^{1/2}\nabla u\big\|^2_{L^2(\Omega_{R+1})} + k^2 \big\|\coeffn^{1/2} v\big\|^2_{L^2(\Omega_{R+1})} \leq \left(\frac{2\sqrt{2}}{\pi}L\big(\coeffn\coeffA^{-1},\Omega, R+2\big)\right)^2\frac{1}{\coeffn_{\min}} \N{f}^2_{L^2(\Omega_+)} ,
\eeq
for all $k\geq k_0$. Since the $\chi$ in our argument can be chosen independent of $\coeffA$ and $\coeffn$, the constant $k_0$ can be chosen uniformly as $\coeffA$ and $\coeffn$ vary
within a sufficiently small open neighborhood in $\CR^{2,\alpha}$ ($\alpha>0$), just as in Theorem \ref{thm:2}.

We then apply 
\eqref{eq:H2} with $v=u$, $\widetilde{f}= -k^2 \coeffn u-f$, and recall that $\supp f \subset \Omega_R$, to obtain
\beqs
\N{u}_{H^2(\Omega_R)} \leq \Creg\left( \frac{2\sqrt{2}}{\pi (\coeffn_{\min})^{1/2}}L\big(\coeffn\coeffA^{-1},\Omega,R+2\big)\left( k(\coeffn_{\max})^{1/2}+ 1 + \frac{1}{k(\coeffn_{\min})^{1/2}} \right) + 1
\right) \N{f}_{L^2(\Omega_R)}.
\eeqs
The result \eqref{eq:H2final} then follows by noting that $L(\coeffn\coeffA^{-1}, \Omega, R+2)\geq 2$; this last inequality holds since the geodesic needs to at least travel from the boundary of the ball of radius $R+2$ into the ball of radius $R$ and out again (a distance of at least $4$), and in this annular region the metric is Euclidean and the flow has speed 2.
\epf

We can now prove Theorem \ref{thm:Emain}.

\bpf[Proof of Theorem \ref{thm:Emain}]
Using the polynomial approximation result \eqref{eq:E2} and the definitions of $\eta(\cH_\hFEM)$ \eqref{eq:etadef} and $\|\cdot\|_{\HoDkkg}$ \eqref{eq:1knormg}, we have
\beqs
\eta(\cH_\hFEM)\leq \Cint \hFEM 
\sqrt{1 + (\hFEM k)^2}
\sup_{0\neq f\in L^2(\Omega_R)}\frac{\N{S^*f}_{H^2(\Omega_R)}}{\N{f}_{L^2(\Omega_R)}}.
\eeqs
By Lemma \ref{lem:adjoint},  the solution of the adjoint exterior Dirichlet problem with data $f$ is the complex conjugate of the solution of the exterior Dirichlet problem with data $\overline{f}$. Therefore, 
the bound \eqref{eq:H2final} 
for the solution of the exterior Dirichlet problem also holds for the solution of the adjoint problem, and implies that
\begin{align*}
\eta(\cH_\hFEM) &\leq  
\Cint \hFEM 
\sqrt{1 + (\hFEM k)^2}
\,\,k\,\Creg \frac{2\sqrt{2}}{\pi (\coeffn_{\min})^{1/2}} \,L\big(\coeffn\coeffA^{-1},\Omega, R+2\big) 
\\ 
&\hspace{5cm}\times\left( (\coeffn_{\max})^{1/2} + \frac{1+ (\coeffn_{\min})^{1/2}}{k_0}+ \frac{1}{k_0^2(\coeffn_{\min})^{1/2}};
\right)
\end{align*}
the result then follows from using this bound on $\eta(\cH_\hFEM)$ in Lemma \ref{lem:Schatz}. 
\epf

\bre[How the bound \eqref{Sobolevupperbound2} can be used in the analysis of preconditioning strategies]\label{rem:precondition}

\

\noi The Galerkin method \eqref{eq:EDPFEM} is equivalent to a linear system of equations; denote the matrix of this linear system by $\matrixA$. Linear systems involving $\matrixA$ are difficult to solve because (a) the dimension of $\matrixA$ is proportional to $\hFEM^{-n}$ and (by Theorem \ref{thm:Emain}) $\hFEM$ must decrease like $k^{-2}$ for the Galerkin solution to be quasioptimal, therefore $\matrixA$ is large, and (b) since $\matrixA$ is large and sparse (when using standard piecewise-polynomial bases of $\cH_\hFEM$), iterative methods are usually used to solve the linear system, but $\matrixA$ is both  non-normal (in general) and sign-indefinite when $k$ is sufficiently large, and the efficient iterative solution of linear systems involving such matrices is difficult.
One therefore seeks to \emph{precondition} $\matrixA$; i.e., to find a $\matrixB$ such that (i) $\matrixB^{-1}$ approximates $\matrixA^{-1}$ and (ii) the action of $\matrixB^{-1}$ is cheap to compute, and one then applies the iterative solver to $\matrixB^{-1}\matrixA$. 

A very popular and successful preconditioner for $\matrixA$ is based on choosing $\matrixB^{-1}$ to be a cheap approximation of $(\matrixA_\absorb)^{-1}$, where $\matrixA_\absorb$ is the Galerkin matrix arising from the exterior Dirichlet problem with $k^2 \mapsto k^2 + \ri \absorb$, i.e.,~with artificial absorption $\absorb$ added.
The rationale behind this method (introduced in \cite{ErOoVu:06}) is that the larger $\absorb$ is, the less oscillatory the problem is, and hence the easier it is to find a cheap approximation to $(\matrixA_\absorb)^{-1}$. However, the larger $\absorb$ is, the further $(\matrixA_\absorb)^{-1}$ is from $\matrixA^{-1}$, hence the question of what is the optimal $\absorb$ is nontrivial. 

Although one of the advantages of preconditioning with absorption, compared to other Helmholtz-preconditioning strategies, is its easy applicability to variable-coefficient problems (see, e.g., \cite{OoVuMu:10}), the only existing analysis is for constant-coefficient Helmholtz problems (i.e.~\eqref{eq:PDE} with $\coeffA=I$ and $\coeffn=1$).
As a component of determining the optimal $\absorb$, \cite[Theorem 1.4]{GaGrSp:15} proved that, when $\coeffA=I$,  $\coeffn=1$, $\Omega$ is starshaped, and a quasi-optimal error estimate (similar to \eqref{eq:QO}) holds for Galerkin discretizations of the problem with absorption, there exists $C>0$ (independent of $k$ and $\hFEM$, but dependent on $\Omega$ and $\cT_\hFEM$) such that 
\beq\label{eq:GGS}
\| \matrixI - (\matrixA_\absorb)^{-1} \matrixA\|\leq C \absorb/k;
\eeq
i.e.~choosing $\absorb$ to be a sufficiently small multiple of $k$ guarantees that $(\matrixA_\absorb)^{-1}$ is a good approximation to $\matrixA^{-1}$, uniformly as $k\rightarrow \infty$.
Using the bound \eqref{Sobolevupperbound2} in the arguments of \cite{GaGrSp:15} one can show that the bound \eqref{eq:GGS} holds when $\coeffA$, $\coeffn$, and $\Omega$ satisfy the conditions of Theorem \ref{thm:Emain} with $C$ a constant multiple of $L(\coeffn\coeffA^{-1},\Omega,R+1)$; i.e., as the length of the longest ray increases, less absorption is allowed for $(\matrixA_\absorb)^{-1}$ to be a good approximation to $\matrixA^{-1}$.
This result is then consistent with the numerical experiments in \cite[Tables 8 and 9]{GaGrSp:15}: 
the geometry corresponding to Table 9 supports longer rays than the geometry corresponding to Table 8, and the numerical results in the tables (the number of iterations required to solve $(\matrixA_\absorb)^{-1}\matrixA$ with GMRES) show that a lower amount of absorption is allowed in the former case than in the latter for $(\matrixA_\absorb)^{-1}$ to be a good preconditioner for $\matrixA$.
\ere

\bre[How the bound \eqref{Sobolevupperbound2} can be used in ``Uncertainty Quantification"]\label{rem:UQ}
In the last 10 years there has been a surge of interest in ``Uncertainty Quantification (UQ)" of PDEs, understood as theory and algorithms for computing statistics of quantities of interest involving PDEs \emph{either} posed on a random domain \emph{or} having random coefficients. 

There is a large literature on UQ for the Poisson equation 
\beq\label{eq:diffusion}
-\nabla\cdot (A(\omega) \nabla u(\omega))=f(\omega),
\eeq
(where $\omega$ is an element of the underlying probability space) due, in part, to its large number of applications (e.g.~in modelling groundwater flow).
The fact that a priori bounds on the solution of \eqref{eq:diffusion} that are explicit in the coefficient $A$ can easily be obtained is the starting point for the rigorous analysis of UQ algorithms; see e.g.~\cite{BaTeZo:04,BaNoTe:07,Gi:10,MuSt:11,Ch:12, ChScTe:13}. For example, for \eqref{eq:diffusion} posed in a bounded Lipschitz domain $D$ with homogeneous Dirichlet boundary conditions and $A\in L^\infty(D)$ with $A_{\min}>0$ 
(in the sense of quadratic forms as in \eqref{eq:Amin}), the Lax-Milgram theorem implies that
\beq\label{eq:diffusion_bound}
\N{u}_{H^1(D)}\leq \frac{C_D}{A_{\min}}
\N{f}_{L^2(D)},
\eeq
where $C_D$ is the constant appearing in the Poincar\'e inequality $\|v\|_{H^1(D)}\leq C_D^{1/2} \|\nabla v\|_{L^2(D)}$ for all $v\in H^1_0(D)$.
In contrast, there has been essentially no rigorous theory of UQ for the Helmholtz equation with large $k$ because a priori bounds on the solution that are explicit in both $k$ and the coefficients have not been available. 

The recent paper \cite{PeSp:18} presented general measure-theory arguments that convert a bound on the (deterministic) Helmholtz equation that is explicit in both $k$ and the coefficients
into a bound, and associated well-posedness result, on the Helmholtz equation with random coefficients (and random data). 
The paper \cite{PeSp:18} used as input to these general arguments the deterministic bounds from \cite{GrPeSp:18} for star-shaped Lipschitz $\Omega$ and coefficients satisfying radial monotonicity-like conditions which are stronger than nontrapping. 
We highlight that the bound \eqref{Sobolevupperbound2}  can be used with the arguments of \cite{PeSp:18} to obtain well-posedness results and a priori bounds on the Helmholtz equation \eqref{eq:PDE} where the coefficients and domain are such that the problem is almost surely nontrapping.
\ere


\end{document}